\documentclass[12pt]{amsart}
\usepackage[top=1in, bottom=1in, left=1in, right=1in]{geometry}
\usepackage{dsfont}
\usepackage{amsmath}
\usepackage{amsthm}
\usepackage{amssymb}
\usepackage{multirow}
\usepackage{mathtools}
\usepackage{times}
\usepackage[utf8]{inputenc}
\usepackage[english]{babel}

\usepackage{indentfirst}  
\usepackage{graphicx}
\usepackage{color}
\usepackage[colorlinks=true]{hyperref}
\hypersetup{colorlinks=true, citecolor=green, linkcolor=green, filecolor=magenta, urlcolor=cyan}
\usepackage[shortlabels]{enumitem}
\usepackage{titlesec}
\titleformat{\chapter}[display]{\bfseries\huge}{\filright\huge\chaptertitlename~\thechapter}{3ex}{\titlerule\vspace{1ex}\filright}[\vspace{1ex}\titlerule]

\newtheorem{thm}{Theorem}[section]
\newtheorem{prop}[thm]{Proposition}
\newtheorem{lem}[thm]{Lemma}
\newtheorem{cor}[thm]{Corollary}

\theoremstyle{definition}

\theoremstyle{definition}

\theoremstyle{remark}
\newtheorem{remark}[thm]{Remark}

\newenvironment{feqn*}{\begin{mdframed}\begin{equation*}}{\vspace{1mm}
\end{equation*}\end{mdframed}}

\numberwithin{equation}{section}

\newcommand{\N}{\mathbb{N}}
\newcommand{\Z}{\mathbb{Z}}

\newcommand{\R}{\mathbb{R}}

\newcommand{\CA}{\mathcal{A}}

\newcommand{\CB}{\mathcal{B}}

\newcommand{\CP}{\mathcal{P}}

\newcommand{\CS}{\mathcal{S}}

\newcommand{\M}{\mathfrak{M}}
\newcommand{\m}{\mathfrak{m}}

\newcommand{\bs}\boldsymbol{}
\renewcommand{\geq}{\geqslant}
\renewcommand{\leq}{\leqslant}
\renewcommand{\hat}{\widehat}
\renewcommand{\tilde}{\widetilde}

\renewcommand{\mod}[1]{\,({\rm mod}\,#1)}
\newcommand{\vect}[1]{\overrightarrow{\boldsymbol{#1}}}

\definecolor{blue}{rgb}{.2,.6,.75}
\definecolor{green}{rgb}{.4,.7,.4}
\definecolor{red}{rgb}{1,0,0}

\titleformat{\section}[block]{\scshape\centering}{\arabic{section}.}{1ex}{}{}
\titleformat{\subsection}[block]{\scshape\bfseries}{\thesubsection}{1ex}{}{}

\begin{document}




\title[Missing digits and sums of two prime squares]{Missing digits and sums of two prime squares}

\author{Cihan Sabuncu}

\address{Max Planck Institute for Mathematics \\
Vivatsgasse 7 \\
53111 Bonn \\
Germany}
\email{sabuncu@mpim-bonn.mpg.de}


\subjclass[2010]{}

\date{\today}

\begin{abstract}
We investigate integers whose base $g$ expansion omits a fixed digit and which can be represented as a sum of two prime squares. In the first part of the paper, we apply the Hardy--Littlewood circle method to obtain asymptotic formulas for weighted count of representations of such integers up to $g^k$ as $k\to\infty$, where we weight by the von Mangoldt function. In this case, we also get an interesting bias depending on the fixed digit we are missing. In the second part, combining the circle method with sieve methods, we study the second moment of the corresponding unweighted counting function. This allows us to get a nontrivial lower bound for the cardinality of the set
$$\{ n \leq g^k : n \text{ omits the digit } b \text{ in its base } g \text{ expansion and } n = p^2 + q^2 \text{ for some primes } p,q \}. $$
\end{abstract}

\maketitle

\section{Introduction}
Numbers with missing digits have been studied extensively in number theory because of their distinctive arithmetic properties. This is particularly true in sieve theory, where such sets are of interest due to their sparsity. In recent years, missing digit problems have also been approached using the Hardy--Littlewood circle method. This approach is viable because the indicator function of the set of missing digit numbers admits strong Fourier decay properties. Exploiting this structure, Maynard \cite{MR4452438} showed that for a sufficiently large base $g$, there are infinitely many primes whose base $g$ expansion omits a given digit $0\leq b<g$. He later refined this result to work for the base $g=10$, see \cite{MR3958793}.

Following the work of Maynard \cite{MR4452438}, there has been substantial progress on problems concerning primes with missing digits. In particular, attention has turned to variants of the classical problems in number theory incorporating missing digit restrictions. One such example is the Vinogradov three primes theorem. In this direction, Maier and Rassias \cite{MR4546880} established, under the Generalized Riemann Hypothesis, a version of Vinogradov’s three primes theorem involving one missing digit prime and two Piatetski–Shapiro primes. This was later refined by Leng and Sawhney \cite{MR4855348}, who proved Vinogradov’s three primes theorem with all three primes missing a digit.

Another central problem in number theory, particularly in the context of the Hardy--Littlewood circle method, is Waring's problem. Given a positive integer $n$, the problem asks for the smallest number of $k$-th powers required to represent $n$ as their sum. This question has been studied extensively, see Vaughan and Wooley \cite{MR1956283} for a survey on this problem.

In a variant of this problem where the $k$-th powers are required to come from numbers with restricted digit expansions, Green \cite{MR4922764} proved that every sufficiently large integer can be written as a sum of at most $ g^{160k^2} $ terms of the form $x^k$, where $x$ has a base $g$ expansion using only two digits.

In this paper, we will study a problem that is somewhat related to both the Waring problem and the missing digit problem. For this, we start by defining the normalized representation function for sums of two prime squares
$$ r_2(n) := \sum_{a^2+b^2=n} \Lambda(a)\Lambda(b), $$
where $\Lambda(\cdot)$ is the von Mangoldt function.
Then we know by the Prime Number Theorem in arithmetic progressions, for any $A>0$, $1\leq q\leq (\log x)^A$ and $0\leq a< q$,
$$ \sum_{\substack{n\leq x \\ n\equiv a \mod q}} r_2(n) =  \frac{\rho(a;q)}{\varphi^2(q)} \cdot \frac{\pi}{4} x + O_A \bigg( \frac{x}{(\log x)^A}\bigg), $$
where 
\begin{equation}\label{def_rho}
\rho(a;q):=\#\{ (u,v)\in (\Z/q\Z)^2 : (uv,q)=1,  u^2+v^2\equiv a \mod q \}.
\end{equation}
Our first theorem is about the average of this function over the missing digit set. For this, we fix the base $g\geq 2$, and let $X=g^k$ for some $k$ large. We take the digit we want to miss $0\leq b < g$. Now, we can define
\begin{equation}\label{def_A}
\CA := \{n \in \N : n=\sum_{i=0}^{\infty} a_i g^i , a_i\not=b\} \quad \text{ and } \quad \CA(X) := \CA \cap [1,X]. 
\end{equation}
\begin{thm}\label{averaging_result}
Fix $g\geq 2$ sufficiently large. Let $A>0$ be given. We have as $k\to\infty$, for $X=g^k$,
\begin{equation}\label{average_sum}
\sum_{n\in \CA(X)} r_2(n) = \mathfrak{S}(b,g) \cdot \frac{\pi}{4} \#\CA(X) + O_{A}\bigg( \frac{\#\CA(X)}{(\log X)^A} \bigg) ,
\end{equation}
where 
$$ \mathfrak{S}(b,g)=\frac{g}{g-1} \bigg( 1 - \frac{\rho(b;g)}{\varphi^2(g)} \bigg) ,$$
and $\rho(a;q)$ is given by \eqref{def_rho}. 
\end{thm}
In this regard, there has been some work on averages of arithmetic functions over missing digit numbers. More recently, Nath \cite[Theorem 7]{MR4680227} proved that for arithmetic functions $f$ with $|f| \ll \log X$ satisfying certain natural hypotheses, one gets a general theorem on the level of distribution of $f(n)$. More precisely, he shows that for $|\sigma|\leq \tau$ an arithmetic function supported on $[1,X^{1/2})$, $c_d$ some congruence class, and
$$\lambda_d = \lim_{y\to\infty} \frac{1}{y}\sum_{\substack{n\leq y \\ (n,d)=1}} f(n),$$
the quantity
$$\sigma(d) \bigg(\sum_{\substack{n\in \mathcal{A}(X) \\ n\equiv c_d \mod d}} f(n)- \frac{\lambda_d}{\varphi(d)}\frac{b}{\varphi(b)} \#\mathcal{A}(X) \bigg)  , $$
is small on average over moduli $d < X^{1/2}$.

\begin{remark}
We note that the local factor $\mathfrak{S}(b,g)$ is maximized when $\exists p| g$ such that $p\equiv 3\mod 4$, and $b=0$, for which, we have $\mathfrak{S}(b,g)=g/(g-1)$. It is minimized when $\forall p|g$ we have $p\equiv 1 \mod 4$ and $b=0$, for which $\mathfrak{S}(b,g)= \frac{g}{g-1} (1 - 2^{\omega(g)}/\varphi(g))$ where $\omega(g)$ is the number of distinct prime factors of $g$. See Lemma \ref{rho_function_properties} for values of $\rho(a;q)$. Moreover,
$$ \frac{1}{g} \sum_{b\mod g} \mathfrak{S}(b,g) = 1 ,$$
\textit{i.e.} the local factor is $1$ on average over the digit $b\mod g$.
\end{remark}
Now, we can define the regular representation function for sums of two prime squares
$$ r_*(n) := \#\{(p,q)\in \CP^2 : p^2+q^2=n \}, $$
where $\CP$ is the set of prime numbers.
This function has been studied before in the works \cite{MR2418801, MR1833070, Erdos1938, MR4756119, MR4389024}. Hence, we can use Theorem \ref{averaging_result} along with some counting arguments to get a bound for this function.
\begin{cor}\label{main_cor}
Fix $g\geq 2$ sufficiently large. We have as $k\to\infty$, for $X=g^k$,
$$\sum_{n\in \CA(X)} r_*(n) \asymp \frac{\#\CA(X)}{(\log X)^2}. $$
\end{cor}
It is possible to get an asymptotic of the form
$$ \sum_{n\in \CA(X)} r_*(n) \sim \mathfrak{S}(b,g) \cdot \pi \frac{\#\CA(X)}{(\log X)^2}, $$
by partial summation, and obtaining an asymptotic to the more general problem
$$ \sum_{\substack{p \leq A, q\leq B \\ p^2+q^2 \in \CA(X)}} 1 , $$
for any $A,B\leq X$. However, we do not pursue this direction in the paper, as our current bounds are insufficient to yield an asymptotic result for Theorem \ref{anatomy_result}. Nonetheless, we expect that this can be proved with some modifications of the ideas in the paper.

We know from earlier results that $r_*(n)$ is $2$ almost surely whenever $r_*(n)>0$, \textit{i.e.}
$$ \#\{ n\leq x : r_*(n) > 0 \} \sim \frac{\pi}{2} \frac{ x}{(\log x)^2} . $$
We want to show a similar result when $n$ is in the set of missing digits. To show such a result, we need to study the second moment of $r_*(n)$ over $n\in \CA(X)$. As an immediate bound, we have
$$ \sum_{n\in \CA(X)} r_*^2(n) \ll_{\varepsilon} \frac{\#\CA(X)^{1+\varepsilon}}{(\log X)^2}. $$
Nonetheless, we can refine this bound by combining the machinery of the circle method with sieve estimates, which leads to the following theorem.
\begin{thm}\label{off-diagonal_result}
Fix $g\geq 2$ sufficiently large. We have as $k\to\infty$, for $X=g^k$,
$$ \sum_{n\in \CA(X)} (r_*^2(n) - 2r_*(n)) = \sum_{\substack{p_1^2+q_1^2=p_2^2+q_2^2 \in \CA(X) \\ \{p_1,q_1\}\not=\{p_2,q_2\}}} 1 \ll \frac{\#\CA(X) (\log\log X)^4}{\log X} . $$
\end{thm}
\begin{remark}\label{off-diagonal-remark}
We actually expect the size of this set to be $\asymp \#\CA(X)/(\log X)^3$. Unfortunately, we lose a $(\log X)^2$ factor mainly because we can't take the major arcs thin enough. This stems from our failure to understand $L^{\infty}$ bounds for the Fourier transform of the indicator function of $\CA$. See Remark \ref{rmk_on_L_infty_bound} for more details.
\end{remark}
Trivially, we have the bound
$$ \#\{ n\in \CA(X) : r_*(n) > 2 \} \ll \frac{\#\CA(X)}{(\log X)^2}, $$
coming from Corollary \ref{main_cor}. Nonetheless, we can say more by using the following inequality,
$$ \#\{n\in \CA(X) : r_*(n) > 2\} \leq \frac{1}{3} \sum_{n\in \CA(X)} (r_*(n) - 2r_*(n)) +O_{\varepsilon}(X^{1/2+\varepsilon}),$$
since $1_{m>2} \leq m(m-2)/3$ for $m>1$, where the big $O$ term comes from the $r_*(n)=1$ contribution. Thus, as stated in Remark \ref{off-diagonal-remark}, we actually expect
$$ \#\{n\in \CA(X) : r_*(n) > 2\} \ll \frac{\#\CA(X)}{(\log X)^3}, $$
in other words, for a typical $n\in \CA(X)$ with $r_*(n)>0$, we have $r_*(n)=2$.

Finally, putting together everything, we will get a result on the quantity of the integers in $\CA$ that can be represented as a sum of two prime squares.
\begin{thm}\label{anatomy_result}
Fix $g\geq 2$ sufficiently large. We have as $k\to\infty$, for $X=g^k$,
$$ \frac{\#\CA(X)}{(\log X)^3 (\log\log X)^4} \ll \#\{ n\in \CA(X) : r_*(n) >0 \} \ll \frac{\#\CA(X)}{(\log X)^2}. $$
\end{thm}
As one would guess, we actually expect
$$ \#\{ n\in \CA(X) : r_*(n) >0 \} \sim \frac{\pi}{2}\cdot \mathfrak{S}(b,g) \frac{\#\CA(X)}{(\log X)^2}, $$
where 
$$ \mathfrak{S}(b,g)=\frac{g}{g-1} \bigg( 1 - \frac{\rho(b;g)}{\varphi^2(g)} \bigg) ,$$
and $\rho(a;q)=\#\{ (u,v)\in (\Z/q\Z)^2 : (uv,q)=1,  u^2+v^2\equiv a \mod q \}$ are as before. Unfortunately, we can't accomplish this because of the loss of the $(\log X)^2$ factor in Theorem \ref{off-diagonal_result}.

Our theorems would also work in the case we are missing more digits. That is, for $0< t <g^{1/768 - \varepsilon}$, and $b_1,\cdots , b_t\in \{0,\cdots , g-1\}$, if we define the set $\CB:= \{n\in \N : n=\sum_{i=0}^{\infty} a_i g^i, a_i\not\in \{b_1,\cdots , b_t\} \}$, and similarly $\CB(X):=\CB \cap [1,X]$, then we can get similar asymptotics and bounds for Theorems \ref{averaging_result}, \ref{off-diagonal_result}, \ref{anatomy_result} and Corollary \ref{main_cor} with $\CB(X)$ instead of $\CA(X)$, and
$$ \mathfrak{S}(\vect{b},g) := \frac{g}{g-t} \bigg( 1  - \sum_{i=1}^t \frac{\rho(b_i ; g)}{\varphi^2(g)} \bigg), $$
instead of $\mathfrak{S}(b,g)$. These results would follow from the modifications mentioned in \cite{MR4452438} and playing around with other constants. The number $768=192\times 4$ comes from Lemma \ref{trigonometric_sum_over_primes}, more precisely \eqref{worst_bound}.

\begin{remark}
We can also prove similar results as Theorems \ref{averaging_result}, \ref{off-diagonal_result}, \ref{anatomy_result} and Corollary \ref{main_cor} for the functions
$$ r_1(n) := \sum_{a^2+b^2=n} \Lambda(a) ,$$
and
$$ \tilde{r_{*}}(n) := \#\{ (p,b)\in \CP\times \N : p^2+b^2=n \}. $$
The proof of these would be very similar with some minor modifications throughout. For example, as an asymptotic, we would get
$$ \sum_{n\in \CA(X)} r_1(n) = \tilde{\mathfrak{S}}(b,g) \frac{\pi}{4} \#\CA(X) + O_A\bigg( \frac{\#\CA(X)}{(\log X)^A} \bigg) , $$
where
$$  \tilde{\mathfrak{S}}(b,g) := \frac{g}{g-1} \bigg( 1  - \frac{\tilde{\rho}(b;g)}{g\cdot \varphi(g)}\bigg), $$
and $\tilde{\rho}(a;q):=\#\{(u,v)\in (\Z/q\Z)^2 : (u,q)=1 , u^2+v^2\equiv a \mod q \}$. Also, we would get
$$ \sum_{n\in \CA(X)} (\tilde{r_{*}}^2(n) - \tilde{r_{*}}(n)) \ll \#\CA(X) (\log\log X)^2 , $$
and
$$ \frac{\#\CA(X)}{(\log X)^2 (\log\log X)^2} \ll \#\{n\in \CA(X) : \tilde{r_{*}}(n) >0\} \ll \frac{\#\CA(X)}{\log X}. $$
However, in this case, obtaining an asymptotic for the size of this set is more difficult and would require studying the second moment of $\tilde{r_{*}}(n)$ over $n\in \CA(X)$ with a fixed number of prime factors. See \cite{MR1833070, MR4772301} for more information on $\tilde{r_{*}}(n)$ and $\tilde{\rho}(a;q)$.
\end{remark}

\subsection{Idea of the Proof}
To prove Theorem \ref{averaging_result}, we use the circle method as was used in the paper of Maynard \cite{MR4452438} to study primes with missing digits. Major arcs are defined by
$$ \bigcup_{s\leq (\log X)^B} \bigcup_{\substack{ 1 \leq r < s \\ (r,s)=1}} \bigg[ \frac{r}{s} - \frac{(\log X)^B}{X} , \frac{r}{s} + \frac{(\log X)^B}{X} \bigg] ,$$
for some $B>0$. Most of the analysis for major arcs is similar except for getting the local factor, which requires understanding the prime factorization of the base $g$. For the minor arcs, we need to understand sums of the form
$$ \sum_{p\leq y} e(\alpha p^2) ,$$
for $\alpha$ away from rationals with small denominator and $y\leq \sqrt{X}$. We can get cancellations for these sums, see Lemma \ref{trigonometric_sum_over_primes}.

For the proof of Theorem \ref{off-diagonal_result}, we use an idea of the author \cite{MR4756119} to split $p_1^2+q_1^2=p_2^2+q_2^2$ in $\Z[i]$ as $(a^2+b^2)(c^2+d^2)$ with some linear forms in $a,b,c,d$ being prime. Then we can go from primes to very rough integers\footnote{We will take as rough as $\exp(C(\log X)^{1/2})$ for some constant $C>0$ depending on $g$.}, and use an upper bound sieve to get rid of the roughness condition, see Lemma \ref{sieve_lemma}. At the end, we get a sum of the form
$$ \sum_{\substack{X^{3/4}/(a^2+b^2) < c^2+d^2 \leq X/(a^2+b^2) \\ (a^2+b^2)(c^2+d^2)\in \CA \\ (c,d)\equiv (v_1,v_2)\mod h}} 1, $$ 
for some $h\leq H$ with $H=\exp(O_g(\log X)^{1/2})$. Then we can use the circle method machinery to understand this sum. In this case, we take major arcs a bit thinner, more precisely
$$ \bigcup_{s\leq \exp(C_g(\log X)^{1/2})} \bigcup_{\substack{1\leq r < s \\ (r,s)=1}} \bigg[ \frac{r}{s} - \frac{\exp(C_g(\log X)^{1/2})}{X} , \frac{r}{s} + \frac{\exp(C_g(\log X)^{1/2})}{X} \bigg] ,$$
for some constant $C_g>0$. However, in this case the local factor is more intricate than in the first part. Since we are only concerned with establishing an upper bound, many of these complications can be safely ignored. A detailed treatment is provided at the end of the paper. For minor arcs, we leverage the extra average over $a^2+b^2$. Essentially, we prove an averaged double quadratic exponential sum estimate, see Lemma \ref{quadratic_exponential}, which gives us what we want.

Lastly, Theorem \ref{anatomy_result} follows from Corollary \ref{main_cor} and Theorem \ref{off-diagonal_result} by the Cauchy-Schwarz inequality and Markov's inequality,
$$ \bigg(\sum_{n\in\CA(X)} r_*(n) \bigg)^2 \bigg/ \bigg( \sum_{n\in \CA(X)} r_*^2(n)\bigg) \leq \#\{ n\in \CA(X) : r_*(n)>0\} \leq \sum_{n\in \CA(X)} r_*(n) ,$$
and splitting the second moment into diagonal and off-diagonal contributions.

Getting the corollaries are easy in comparison, and we leave the fine details to the reader. To get Corollary \ref{main_cor}, we use the following inequality
$$ (\log X^{1/4})^2 \sum_{\substack{p^2+q^2 \in \CA(X) \\ p,q>X^{1/4}}} 1 < \sum_{n\in \CA(X)} r_2(n) \leq (\log \sqrt{X})^2 \sum_{n\in \CA(X)} r_*(n). $$
Next, we note that
$$ \sum_{\substack{p^2+q^2 \in \CA(X) \\ p,q>X^{1/4}}} 1 = \sum_{n\in \CA(X)} r_*(n) + O(X^{3/4}), $$
which proves the result.

\textbf{Notation:} We use the usual asymptotic notation of Vinogradov. By $(a,b)\equiv (u,v) \mod q$ we mean $a\equiv u \mod q$ and $b\equiv v \mod q$. Throughout the paper, $g$ is the base sufficiently large and $b$ is the digit we are missing. We have $X:=g^k$ with $k$ an integer going to infinity. We use $(a,b)=1$ to mean that $a$ and $b$ are coprime, and use $[a,b]$ for the least common multiple of $a$ and $b$. All our bounds depend on the base $g$ which we only mention to remind the reader when necessary. We also use $\ll_A (\log X)^{-A}$ or $O_A((\log X)^{-A})$ to mean that such a bound holds for any $A$ with a constant depending on $A$, and similarly $\ll_{\varepsilon} X^{\varepsilon}$. In Section \ref{average_section}, we will have a fixed $B$ that is sufficiently large which will take the place of $A$ to not confuse the appearance of the other $A$'s. In Section \ref{off_diagonal_section}, we will use $\CP_g$ to be the set of primes coprime to $3g$.

\textbf{Acknowledgements.}
The author is grateful to Fei Xue for his discussion concerning missing digits and explaining some basic ideas behind them. He would also like to thank Valentin Blomer, Pieter Moree, Kunjakanan Nath, and Alisa Sedunova for their comments and suggestions. Furthermore, he wishes to thank Jeremy Schlitt for reading an earlier version of the manuscript.

\section{Preliminary Lemmas}

\subsection{Elementary Results}

We start by studying the properties of the function $\rho(a;q)$ given by \eqref{def_rho}.
\begin{lem}\label{rho_function_properties}
We have for $p\geq 3$,
$$ \rho(\nu ; p) = \begin{cases} 
      (1+\chi_4(p))(p-1) & \nu\equiv 0 \mod p \\
      p-1-\chi_4(p)-\big(\frac{\nu}{p}\big) & \nu\not\equiv 0\mod p
   \end{cases} $$
where $\chi_4$ is the non-principal character modulo $4$, and $\rho(\nu;p^a) = p^{a-1} \rho(\nu; p)$ for $a\geq 1$\footnote{Here it is understood $\rho(\nu;p)$ refers to $\rho(\nu\mod p;p)$.}. In the case $p=2$, we have $\rho(\nu;2)= 1_{\nu\equiv 0\mod 2}$, $\rho(\nu;4)=1_{\nu\equiv 2 \mod 4}\cdot 4$ and $\rho(\nu;2^a)= 2^{a-3} \rho(\nu; 8)=1_{\nu\equiv 2 \mod 8}\cdot 2^{a+1}$ for $a\geq 3$.
\end{lem}
\begin{proof}
This is \cite[Lemma 15]{MR616225}, but we give a short proof here. The proof of the first part comes by classical estimates on the sums of Legendre symbols. To lift to prime powers, fix a $\nu \mod{p^a}$, for a given solution $(u,v)\in (\Z/p^{a-1}\Z)^2$ with $u^2+v^2\equiv \nu \mod{p^{a-1}}$, then we can put $(u+p^{a-1}k,v+p^{a-1}\ell)$ for $0\leq k,\ell<p $. Hence, we have
$$ (u+p^{a-1}k)^2+(u+p^{a-1}\ell)^2 \equiv u^2+v^2 +2 p^{a-1}(k+\ell) \mod{p^a}, $$
and this is $\equiv \nu \mod{p^a}$ if and only if $2(k+\ell) \equiv (\nu-u^2-v^2)/p^{a-1} \mod p$. Therefore, we only have $p$ choices for the pairs $(k,\ell)$. Now, we have $(u_1+p^{a-1}k_1,v_1+p^{a-1}\ell_1)=(u_2+p^{a-1}k_2,v_2+p^{a-1}\ell_2)$ if and only if $u_1\equiv u_2, v_1\equiv v_2, k_1\equiv k_2, \ell_1\equiv \ell_2 \mod{p^{a-1}}$ since $(u_1v_1u_2v_2, p)=1$.
For the powers of $2$, we check by hand $2, 4, 8$ and we know if $\nu\not\equiv 2 \mod 8$, then $\rho(\nu;2^a)=0$ for all $a\geq 3$. To lift $\nu\equiv 2 \mod 8$, we refer to \cite[Lemma]{MR1833070} since both functions are the same in this case.
\end{proof}
Next, we have a lemma on behaviour of the $r_2(n)$ function in short intervals.
\begin{lem}\label{prime_number_theorem_in_short_int}
Let $X$ be a sufficiently large number, and fix $A>0$. We have for $t\leq X$,
$$ \sum_{t<n\leq t + X/(\log X)^A} r_2(n) = \frac{\pi}{4} \frac{X}{(\log X)^A} + O\bigg( \frac{X}{(\log X)^{10A}} \bigg). $$
\end{lem}
\begin{proof}
To prove this, we first assume $t\leq \sqrt{X}$, then we  have
$$ \sum_{t<n\leq t + X/(\log X)^A} r_2(n)= \sum_{n\leq t + X/(\log X)^A} r_2(n) + \sum_{n\leq t} r_2(n), $$
for which, the claim holds trivially.
So, we can assume $t>\sqrt{X}$. We first take care of the case when $n=a^2+b^2$ has $\min\{a,b\} \leq t^{1/4}$. In this case, we trivially have
$$ \sum_{\substack{t<a^2+b^2<t+X/(\log X)^A \\ \min\{a,b\}\leq t^{1/4}}} \Lambda(a)\Lambda(b) \ll_{\varepsilon} X^{1/2}t^{1/4+\varepsilon} \leq X^{3/4+\varepsilon},  $$
since $\#\{a^2+b^2=n\} \ll n^{\varepsilon}$. We focus on the annulus $t < x^2+y^2 \leq t+X/(\log X)^A$ with $\min\{x,y\}>t^{1/4}$, and split this region into a dyadic partition of squares. Let $W$ be a smooth non-negative function supported on $[1,2]$ such that
$$ \sum_{M_1,M_2} W\bigg( \frac{x}{M_1} \bigg)W\bigg( \frac{y}{M_2} \bigg) = 1, $$
where $M_1,M_2\geq 1$ run over a sequence of positive real numbers such that $t\leq M_1^2+M_2^2 \leq t + X/(\log X)^A$, and $\min\{M_1,M_2\} > t^{1/4}$ with  $\#\{M_1,M_2\} \ll (\log X)^2$.
Then we have by the prime number theorem in short intervals and partial summation,
\begin{align*}
\sum_{M_1,M_2} \sum_{m_1,m_2} W\bigg( \frac{m_1}{M_1} \bigg)W\bigg( \frac{m_2}{M_2} \bigg)  &\Lambda(m_1) \Lambda(m_2) = \sum_{M_1,M_2} \int_{\R} W\bigg( \frac{x}{M_1} \bigg) dx \int_{\R}W\bigg( \frac{y}{M_2} \bigg) dy   \\
&+ O\bigg( \sum_{t<M_1^2+M_2^2<t+X/(\log X)^A} \frac{M_1M_2}{\min\{(\log M_1)^{50A} , (\log M_2)^{50A}\}} \bigg),
\end{align*}
where we used $W( u /M ) \ll 1$ for $M\leq u \leq 2M$.
Since we have $\min\{M_1,M_2\} > t^{1/4} \gg X^{1/8}$ since $t>\sqrt{X}$. Using this, we get that the error term is $\ll X/(\log X)^{10A}$ since $t\leq X$. For the main term, we have by changing the sum and the integrals,
$$ \sum_{M_1,M_2} \int_{\R^2} W\bigg( \frac{x}{M_1} \bigg) W\bigg( \frac{y}{M_2} \bigg) dxdy  = \int_{\substack{t<x^2+y^2 <t+X/(\log X)^A \\ \min\{x,y\} > t^{1/4}}} 1 dxdy = \frac{\pi}{4}\frac{X}{(\log X)^A} + O_{\varepsilon}(X^{3/4+\varepsilon}).$$
\end{proof}

The following is an elementary lemma that we give without the proof.
\begin{lem}\label{sum_of_squares_in_short_int_and_aps}
Let $x$ be a sufficiently large number and take $q\leq x^{\varepsilon}$. We have for $0 < u,v \leq q$,
$$ \sum_{\substack{a^2+b^2\leq x \\ (a,b)\equiv (u,v)\mod q}} 1 = \frac{\pi}{4} \frac{x}{q^2}+O\bigg( \frac{\sqrt{x}}{q} \bigg) .$$
\end{lem}

\subsection{Exponential Sum Estimates}
We start with the following result on exponential sums over squares of primes. Similar estimates appeared in \cite{MR1829558} to get results on Waring's problem in higher powers.
\begin{lem}\label{trigonometric_sum_over_primes}
Let $\alpha=\frac{r}{s} +\xi$ with $(r,s)=1$, $s\leq N^{3/4}$, and $s|\xi|<1/N^{3/4}$. Then we have uniformly for $x\leq \sqrt{N}$ and $s \leq N^{1/8}$,
\begin{equation*}
\sum_{n\leq x} \Lambda(n) e(\alpha n^2) \ll_{\varepsilon} \bigg( N^{2/5+\varepsilon} + \frac{1}{s^{1/2-\varepsilon}|\xi|^{1/2}} \bigg) (\log N)^c , \tag{i}
\end{equation*}
for some absolute constant $c>0$, and for $s\geq (\log N)^{3073}$, we have
\begin{equation*}
\sum_{n\leq x} \Lambda(n) e(\alpha n^2) \ll N (\log N) \bigg( \frac{1}{s} + \frac{1}{N^{1/3}}+ \frac{s}{N^2}  \bigg)^{\frac{1}{192}} . \tag{ii}
\end{equation*}
\end{lem}
\begin{proof}
See Vinogradov \cite{zbMATH02517953} for the latter inequality. For the first inequality, we trivially bound when $x\leq N^{2/5}$ which is negligible. In the case $x> N^{2/5}$, we use \cite[Theorem 2]{MR2252758} by splitting $x^{5/8}<p\leq x$ into dyadic ranges $P<p\leq 2P$ and with $Q=P^{1/2}$. Note that we have $N^{1/4}<P\leq \sqrt{N}$ in this case, and since $s\leq N^{1/8}<P^{1/2}$ the condition $|s\alpha - r| \leq 1/\sqrt{N}< 1/P^{3/2}$ is satisfied by definition. We bound $n\leq x^{5/8} \ll N^{5/16}$ trivially, which can be absorbed to the $N^{2/5+\varepsilon}$ term.
\end{proof}
We use this to get information on exponential sums of $r_2(n)$.
\begin{lem}\label{exponential_sum_of_r_2}
Under the assumptions of Lemma \ref{trigonometric_sum_over_primes}, we have for $s\leq N^{1/8}$,
\begin{equation*}
\sum_{n\leq N} r_2(n) e(\alpha n)    \ll_{\varepsilon} \bigg( N^{9/10+\varepsilon} + \frac{\sqrt{N}}{s^{1/2-\varepsilon}|\xi|^{1/2}} \bigg) (\log N)^c  ,\tag{i}
\end{equation*}
for some absolute constant $c>0$, and for $s\gg (\log N)^{3073}$,
\begin{equation*}
\sum_{n\leq N} r_2(n) e(\alpha n) \ll N (\log N)\bigg( \frac{1}{s}+ \frac{1}{N^{1/3}} + \frac{s}{N^2} \bigg)^{\frac{1}{192}} . \tag{ii}
\end{equation*}
\end{lem}
\begin{proof}
Recall that $r_2(n)= \sum_{n=a_1^2+a_2^2} \Lambda(a_1)\Lambda(a_2)$. Hence, if we write $n=q_1^2+q_2^2$, then by the triangle inequality, we get
$$ \left|\sum_{n\leq N} r_2(n) e(\alpha n) \right| \leq \sum_{q_1\leq \sqrt{N}} \Lambda(q_1) \bigg| \sum_{q_2\leq  \sqrt{N-q_1^2}} \Lambda(q_2) e(\alpha q_2^2) \bigg|, $$
and use Lemma \ref{trigonometric_sum_over_primes}.
\end{proof}
Next, we have some lemmas from Maynard \cite{MR4452438} on the exponential sum over the set of missing digits $\CA$. For this, define
$$ \hat{1}_{\CA}(\theta) := \sum_{n\in \CA(X)} e(\theta n). $$
\begin{lem}[Maynard \protect{\cite[Lemma 5.1]{MR4452438}}]\label{L1_bound_for_fourier_transform}
There exists a constant $C_g \in [1/\log g,1+3/\log g]$ such that
$$ \sup_{\theta\in \R} \sum_{0\leq a < X} \bigg|\hat{1}_{\CA}\bigg(\theta + \frac{a}{X}\bigg)\bigg| \ll (C_g g \log g)^k. $$
\end{lem}

\begin{lem}[Maynard \protect{\cite[Lemma 5.3]{MR4452438}}]\label{hybrid_bound_for_fourier_transform}
Let $S,B\gg 1$. Then we have
$$ \sum_{S\leq s <2S} \sum_{\substack{1\leq r < s \\ (r,s)=1 }} \sum_{\substack{|\eta|< D \\ Xr/s+\eta \in \Z}} \bigg| \hat{1}_{\CA}\bigg( \frac{r}{s} + \frac{\eta}{X} \bigg)\bigg| \ll \#\CA(X) (S^2D)^{\alpha_g} + S^2D(C_g \log g)^k , $$
where $C_g$ is the constant in Lemma \ref{L1_bound_for_fourier_transform}, and
$$ \alpha_g = \frac{\log\big(C_g\frac{g}{g-1}\log g\big)}{\log g} .$$
\end{lem}

\begin{lem}\label{Linfty_bound_for_fourier_transform}
Let $s<X^{1/3}$ be of the form $s=s_1s_2$ with $(s_1,g)=1$ and $s_1\not=1$, and let $|\epsilon|<1/2X^{2/3}$. Then for any $(r,s)=1$, we have
$$ \bigg|\hat{1}_{\CA} \bigg( \frac{r}{s} +\epsilon \bigg)\bigg| \ll \#\CA(X) X^{- c_g /\log s} , $$
for some constant $c_g>0$ depending only on $g$.
\end{lem}
\begin{proof}
This can be obtained from \cite[Lemma 5.4]{MR4452438} by changing $c_q$ there to $c_g /\log g$.
\end{proof}

\begin{remark}\label{rmk_on_L_infty_bound}
Note that if we had major arcs of the form
$$ \bigcup_{s\leq S}\bigcup_{\substack{1 \leq r < s \\ (r,s)=1}} \bigg[ \frac{r}{s} - \frac{S}{X} , \frac{r}{s} + \frac{S}{X}  \bigg], $$
then we would be summing the above quantity over $s\leq S$, which gives
$$ \sum_{s\leq S} \sum_{\substack{1\leq r < s  \\ (r,s)=1}} \sum_{\substack{|\eta|\leq S \\ Xr/s + \eta \in \Z}} \bigg|\hat{1}_{\CA} \bigg( \frac{r}{s} +\varepsilon \bigg)\bigg| \ll \#\CA(X) X^{- c_g /\log S} S^3. $$
From this, one can see that the best $S$ you can hope, to get savings for this sum, is $\exp(c_g^{1/2}/3 (\log X)^{1/2})$. This in turn means our error term can only have a saving of a power of $\exp(O((\log X)^{1/2}))$. However, this means we can't take our sieving variable to be beyond this, and we need $S=\exp((\log X)^{3/4+\varepsilon})$ to get an asymptotic in Theorem \ref{anatomy_result}.
\end{remark}
Finally, we have the following exponential sum to study minor arcs coming from the study of Theorem \ref{off-diagonal_result}.
\begin{lem}\label{quadratic_exponential}
Let $\alpha = r/s + \xi$ with $(r,s)=1$ , $s\leq x$, and $|\xi|< 1/s^2$. Also let $M_1,M_2,N$ be positive integers with $(M_1^2+M_2^2)N\ll x$ and $\max\{ M_1,M_2 \} \gg x^{\delta}$ for some $\delta>0$. Then, for any sequence $(a_n)_{n\in \N}$ with $|a_n|\leq \tau(n)$, we have\footnote{In our application, we won't have the second term since we will have $s|\xi|\leq 1/\sqrt{x}$ and $\delta=1/8$.}
$$\sum_{n\leq N} a_n \sum_{\substack{m_1\leq M_1 \\ m_2\leq M_2 \\ (m_1,m_2)\equiv (u,v)\mod q}} e(\alpha n(m_1^2+m_2^2)) \ll q^{1/2} \bigg( x^{1-\delta/4} + x (s|\xi|)^{1/4} + \frac{x^{3/4}}{(s|\xi|)^{1/4}} \bigg) (\log x)^3 ,$$
and also for $s\gg 1$ sufficiently large
$$\sum_{n\leq N} a_n \sum_{\substack{m_1\leq M_1 \\ m_2\leq M_2 \\ (m_1,m_2)\equiv (u,v)\mod q}} e(\alpha n(m_1^2+m_2^2)) \ll q^{1/2} \bigg( x^{1-\delta/4} +s^{1/4}x^{3/4}+ \frac{x}{s^{1/4}} \bigg) (\log x)^3 . $$
\end{lem}
\begin{proof}
Without loss of generality, we can assume $M_2\leq M_1$, and thus $NM_1^2 \ll x$. We start by applying Cauchy-Schwarz twice to our equation to get
\begin{align*}
\bigg| \sum_{n\leq N} a_n \sum_{\substack{m_1\leq M_1 \\ m_2\leq M_2 \\ (m_1,m_2)\equiv (u,v)\mod q}}& e(\alpha n(m_1^2+m_2^2)) \bigg|  \leq \bigg( M_1M_2 \sum_{n\leq N} |a_n|^2 \bigg)^{1/2} \\
&\hspace{1cm}  \times\bigg( \sum_{n\leq N} \bigg|  \sum_{\substack{m_2\leq M_2 \\ m_2\equiv v\mod q}} e(\alpha nm_2^2) \bigg| \cdot \bigg|  \sum_{\substack{m_1\leq M_1 \\ m_1\equiv u\mod q}} e(\alpha nm_1^2) \bigg| \bigg)^{1/2} \\
&\leq \bigg( M_1M_2 \sum_{n\leq N} |a_n|^2 \bigg)^{1/2} \\
&\hspace{1cm}\times\bigg( \sum_{n\leq N} \bigg|  \sum_{\substack{m_2\leq M_2 \\ m_2\equiv v\mod q}} e(\alpha nm_2^2) \bigg|^2 \bigg)^{1/4} \bigg( \sum_{n\leq N} \bigg|  \sum_{\substack{m_1\leq M_1 \\ m_1\equiv u\mod q}} e(\alpha nm_1^2) \bigg|^2 \bigg)^{1/4} .
\end{align*}
We note that the sum $\sum_{n\leq N} |a_n|^2 \ll N (\log x)^3$ since $|a_n|\leq \tau(n)$. By symmetry, we just focus on the $m_1$ sum. We have, by the Weyl differencing,
$$ \bigg|\sum_{\substack{m_1\leq M_1 \\ m_1\equiv u\mod q}} e(\alpha nm_1^2) \bigg|^2 \leq \sum_{\substack{|h|\leq M_1 \\ q|h}} \bigg| \sum_{\substack{m\leq M_1 \\ m\equiv u \mod q}} e(2\alpha n h m) \bigg| \ll M_1 +  \sum_{0<|h|\leq M_1} \min\{ M_1 , \| 2\alpha q n h \|^{-1} \} . $$
We also have the sum over $n\leq N$, hence, we can glue together $\ell:=2n|h|q$ to get
$$ \sum_{n\leq N}\bigg|\sum_{\substack{m_1\leq M_1 \\ m_1\equiv u\mod q}} e(\alpha nm_1^2) \bigg|^2  \ll NM_1 + \sum_{\ell \leq 2qNM_1} \tau(\ell) \min\{ M_1 , \|\alpha \ell \|^{-1} \} . $$
We split the sum over $\ell$ into $\tau(\ell)\leq B$ and $\tau(\ell) > B$ with
$$ B:= \bigg( \frac{1}{M_1} + s|\xi| + \frac{1}{qNM_1^2 s|\xi|}\bigg)^{-1/2}. $$
Then for $\tau(\ell)\leq B$ using the trivial bound and \cite[Lemma 4.1]{MR4452438}\footnote{Note that we have $d|\xi|<1/d$ and so we don't have the last $d$ term.} for the remaining sum, and for $\tau(\ell)>B$ using Rankin's trick, we get
$$ \sum_{n\leq N} \bigg|\sum_{\substack{m_1\leq M_1 \\ m_1\equiv u\mod q}} e(\alpha nm_1^2) \bigg|^2  \ll NM_1 + qNM_1^2 \bigg( \frac{1}{M_1} + s|\xi| + \frac{1}{NM_1^2 s|\xi|}\bigg)^{1/2} (\log x)^3. $$
Note that the $NM_1$ term is smaller than the first term on the right. We get a similar result for the $m_2$ sum, and since $M_2\leq M_1$, we can assume that we get the same bound. Hence, we get
\begin{align*}
\sum_{n\leq N} a_n \sum_{\substack{m_1\leq M_1 \\ m_2\leq M_2 \\ (m_1,m_2)\equiv (u,v)\mod q}} e(\alpha n(m_1^2+m_2^2)) &\ll q^{1/2} ( N M_1^2)^{1/2} \bigg( N^2M_1^3 + N^2M_1^4 s|\xi| + \frac{NM_1^2}{s|\xi|}  \bigg)^{1/4} (\log x)^3 \\
&\ll q^{1/2} \bigg( N M_1^{7/4} + N M_1^2 (s|\xi|)^{1/4} + \frac{(NM_1^2)^{3/4}}{(s|\xi|)^{1/4}} \bigg) (\log x)^3.
\end{align*}
Lastly, since $NM_1^2 \ll x$ and $M_1\gg x^{\delta}$, we get the first result. For the second one, we use the classical estimate
$$ \sum_{n\leq N} \min\{M , \|\alpha n \|^{-1}\} \ll \bigg(M+N+\frac{MN}{s} + s \bigg) \log x ,$$
instead.
\end{proof}

\subsection{Sieve Result}
Lastly, we have the following beta sieve result, see \cite[Fundamental Lemma 6.8]{MR2061214}.
\begin{lem}\label{sieve_lemma}
Let $z\geq 2$, $\kappa>0$, and $D=z^s$ with $s\geq 9\kappa + 2$, and $\CS$ be a set of primes. There exists coefficients $\lambda_d$ depending only on $\kappa$ and $D$ such that the followings hold

(i) We have that $|\lambda_d|\leq 1$ and is supported in $\{d\leq D : d|P(z)\}$, where $P(z):=\prod_{\substack{p<z \\ p\in \CS}} p$.

(ii) We have 
$$1_{(n,P(z))=1} \leq \sum_{d|n}\lambda_d.$$

(iii) If $g:\N\to [0,1)$ is a multiplicative function such that $g(p)=0$ if $p\not\in \CS$, and
$$ \prod_{w_1 < p \leq z_1} (1 - g(p))^{-1} \ll_{\kappa} \bigg(\frac{\log z_1}{\log w_1}\bigg)^{\kappa}, $$
for any $2\leq w_1 < z_1$, then
$$ \sum_{d|P(z)} \lambda_d \cdot g(d) \ll_{\kappa} \prod_{p\leq z} (1 - g(p) ). $$
\end{lem}

\section{Average of the Representation Function}\label{average_section}
We work on \eqref{average_sum}. By inverse Fourier transform
$$ \sum_{n\in \CA(X)} r_2(n) = \sum_{n\leq X} r_2(n) 1_{\CA}(n) = \frac{1}{X} \sum_{a\leq X} \hat{1}_{\CA}\bigg(\frac{a}{X}\bigg) S\bigg(\frac{-a}{X}\bigg), $$
where
$$ \hat{1}_{\CA}(\theta):=\sum_{n\in \CA(X)} e(\theta n) $$
is transform of the indicator function of the set $\CA$, and
$$ S(\theta):=\sum_{n\leq X} r_2(n) e(\theta n), $$
is the Fourier transform of $r_2(n)$.
Now, we split into major and minor arcs. If we fix $B>0$ large enough, then we can define
$$ \M_{r,s} := \bigg\{ a \in [1,X]\cap \Z : \bigg|\frac{a}{X} - \frac{r}{s} \bigg| \leq \frac{(\log X)^B}{X} \bigg\} . $$
Then, major arcs are defined as
$$ \M := \bigcup_{s\leq (\log X)^B} \bigcup_{\substack{1\leq r <s \\ (r,s)=1}} \M_{r,s}, $$
and minor arcs are
$$ \m := ([1,X]\cap \Z)\setminus \M. $$
Then we have
\begin{align*}
 \frac{1}{X} \sum_{a\leq X} &\hat{1}_{\CA}\bigg(\frac{a}{X}\bigg) S\bigg(\frac{-a}{X}\bigg) = \frac{1}{X} \sum_{a\in \M} \hat{1}_{\CA}\bigg(\frac{a}{X}\bigg) S\bigg(\frac{-a}{X}\bigg)  +  \frac{1}{X} \sum_{a\in \m} \hat{1}_{\CA}\bigg(\frac{a}{X}\bigg) S\bigg(\frac{-a}{X}\bigg) \\
&= \frac{1}{X} \sum_{s\leq (\log X)^B} \sum_{\substack{1\leq r < s \\ (r,s)=1}} \sum_{\substack{|\eta|\leq (\log X)^B \\ Xr/s + \eta\in \Z}}\hat{1}_{\CA}\bigg(\frac{r}{s}+ \frac{\eta}{X}\bigg) S\bigg(\frac{-r}{s} + \frac{-\eta}{X}\bigg) + \frac{1}{X} \sum_{a\in \m} \hat{1}_{\CA}\bigg(\frac{a}{X}\bigg) S\bigg(\frac{-a}{X}\bigg)  ,
\end{align*}
since we can write $\xi=\frac{\eta}{X}$.

\subsection{Major Arcs}
For the major arcs, we have $\frac{a}{X} = \frac{r}{s} + \frac{\eta}{X}$ with $|\eta| \leq (\log X)^B$. We split into 3 cases:

\noindent \textit{\textbf{Case 1. $\exists p|s$ but $p\nmid g$:}} 
In this case by Lemma \ref{Linfty_bound_for_fourier_transform}, we get 
\begin{equation}\label{major_arc_error_case_1}
 \frac{1}{X} \sum_{\substack{s\leq (\log X)^B \\  \exists p|s, p\nmid g}} \sum_{\substack{1 \leq r <s  \\ (r,s)=1}} \sum_{\substack{|\eta|\leq (\log X)^B \\ Xr/s + \eta\in \Z}}\hat{1}_{\CA}\bigg(\frac{r}{s}+ \frac{\eta}{X}\bigg) S\bigg(\frac{-r}{s} + \frac{-\eta}{X}\bigg) \ll _{\varepsilon}\#\CA(X)\cdot \exp( - c_g (\log X)^{1-\varepsilon}) ,
\end{equation}
by trivially bounding $S(-r/s - \eta/X) \ll X$, and using the bound $\log X/(B\log\log X) \gg_{\varepsilon} (\log X)^{1-\varepsilon}$ for $X\to\infty$.

\noindent \textit{\textbf{Case 2. $\eta\not=0$:}} 
In this case, we have
\begin{align*}
\sum_{n\leq X} r_2(n) e\bigg(\bigg(\frac{-r }{s} + \frac{-\eta }{X} \bigg) n \bigg) &= \sum_{\nu\mod s} e\bigg( \frac{-r\nu}{s}\bigg) \sum_{\substack{n\leq X \\ n\equiv \nu \mod s}} r_2(n) e\bigg(\frac{-\eta n}{X}\bigg)\\
&= \frac{1}{\varphi^2(s)}\sum_{\nu\mod s} e\bigg( \frac{-r\nu}{s}\bigg)  \rho(\nu;s) \sum_{n\leq X} r_2(n) e\bigg(\frac{-\eta n}{X}\bigg)+ E( X ;s),
\end{align*}
where we recall $\rho(\nu;s)$ is the multiplicative function
$$ \rho(\nu;s) = \#\{ (a,b)\in (\Z/s\Z)^2 : (ab,s)=1,  a^2+b^2\equiv \nu \mod s \}, $$
and
$$E(X;s) \leq s \cdot \max_{\nu\mod s}\bigg| \sum_{n\leq X} r_2(n) e\bigg( \frac{-\eta n}{X} \bigg) \bigg(1_{n\equiv \nu \mod s} - \frac{\rho(\nu;s)}{\varphi^2(s)}\bigg) \bigg| . $$
By the Siegel-Walfisz theorem, we know 
$$E(t) :=\max_{\nu\mod s}\bigg|\sum_{n\leq t}  r_2(n) \bigg(1_{n\equiv \nu \mod s} - \frac{\rho(\nu;s)}{\varphi^2(s)}\bigg) \bigg| \ll_A \frac{t}{(\log t)^A},$$
so, using partial summation we get that
\begin{align}
E(X;s) &= s\cdot \bigg( e(-\eta) E(X) - e(-\eta/X)E(1) + \frac{\eta}{X}\cdot \int_1^X  e\bigg( \frac{-\eta t}{X} \bigg) E(t) dt \bigg) \nonumber \\
&\ll_A X/(\log X)^A  \label{siegel_walfisz_bound},
\end{align} 
since $s, |\eta| \leq (\log X)^B$. Hence, bounding the Fourier transform $\hat{1}_{\CA}$ trivially and summing over $s,\eta$, we get an error term of $\ll X/(\log X)^B$ coming from $E(X;s)$. By Lemma \ref{prime_number_theorem_in_short_int}, we get by splitting into short intervals,
\begin{align}
\sum_{n\leq X} r_2(n) e\bigg(\frac{-\eta n }{X} \bigg) &= \sum_{t\ll X} \sum_{t < n \leq t + X/\lfloor (\log X)^{5B}\rfloor} r_2(n) e\bigg(\frac{-\eta n }{X} \bigg) \nonumber  \\
&= \sum_{t\ll X} e\bigg( \frac{-\eta t}{X} \bigg) \sum_{t<n\leq t+X/\lfloor (\log X)^{5B} \rfloor} r_2(n) + O_B \bigg( \frac{X}{(\log X)^{4B}} \bigg)  \nonumber \\
&= \frac{\pi}{4} \frac{X}{(\log X)^{5B}} \sum_{t\ll X} e\bigg( \frac{-\eta t}{X} \bigg) + O_B\bigg( \frac{X}{(\log X)^{4B}} \bigg) \nonumber \\
&= \frac{\pi}{4} \frac{X}{(\log X)^{5B}} \sum_{0\leq m < \lfloor (\log X)^{5B}\rfloor}e\bigg( \frac{-\eta m}{\lfloor (\log X)^{5B}\rfloor} \bigg) + O_B \bigg( \frac{X}{(\log X)^{4B}} \bigg) \ll \frac{X}{(\log X)^{4B}} \label{short_interval_bound},
\end{align}
where we parametrized $t=m X/\lfloor (\log X)^{5B}\rfloor$ and also noticed the exponential sum is $0$ since $\eta\not=0$. Thus, from \eqref{siegel_walfisz_bound} and \eqref{short_interval_bound}, we get an error term of
\begin{equation}\label{major_arc_error_case_2}
\ll \frac{(\log X)^{3B}}{X} \#\CA(X) \frac{X}{(\log X)^{4B}} +  \frac{\#\CA(X)}{(\log X)^B} \ll \frac{\#\CA(X)}{(\log X)^{B}}, 
\end{equation}
where we bounded trivially the Fourier transform $|\hat{1}_{\CA}(\theta) | \leq \#\CA(X)$.

\noindent \textit{\textbf{Case 3. $\eta=0$ and $p|s\implies p|g$:}}
This will be the main term, and we focus on the subset of $s\leq (\log X)^B$ with $p|s \implies p|g$, and $\eta=0$. Then we have
\begin{gather*}
\frac{1}{X} \sum_{\substack{s\leq (\log X)^B \\ p|s\implies p|g}} \sum_{\substack{1\leq r < s \\ (r,s)=1}} \frac{1}{\varphi^2(s)} \sum_{c\in \CA} \sum_{\nu \mod s} e\bigg(\frac{-r(\nu-c)}{s} \bigg) \rho(\nu; s) \sum_{n\leq X} r_2(n)
\end{gather*}
We can take out the sum over $n\leq X$ since there is no dependence on it elsewhere. For the remaining sum, we note that for any choice of $J \leq (\log\log X)/\log 2$, we focus on $s|g^J$. So, we have
\begin{gather*}
\frac{\pi}{4} \cdot  \sum_{\substack{s|g^J }} \sum_{\substack{1\leq r < s \\ (r,s)=1}} \frac{1}{\varphi^2(s)} \sum_{c\in \CA} \sum_{\nu \mod s} e\bigg(\frac{-r(\nu-c)}{s} \bigg) \rho(\nu; s) + O_A\bigg( \frac{\#\CA(X)}{(\log X)^A} \bigg).
\end{gather*}
Now, we write $s=s_1\cdots s_{\ell}$ where $g=g_1\cdots g_{\ell}$ where $g_i$ are pairwise coprime prime powers and $s_i | g_i^J$. In this case, we can further write $r=\sum_{i=1}^{\ell}r_i \prod_{j\not=i} s_j$ for some unique $(r_i,s_i)=1$ for $i\in [\ell]$. Similarly, we can write $\nu=\sum_{i=1}^{\ell} \nu_i \prod_{j\not= i} s_j$ for some unique $\nu_i \mod{s_i}$. This gives us
\begin{align*}
\sum_{\substack{s| g^J}}\sum_{\substack{1\leq r < s \\ (r,s)=1}} \frac{1}{\varphi(s)} \sum_{\nu \mod s} \frac{\rho(\nu; s)}{\varphi(s)} &\sum_{c\in \CA}e\bigg(\frac{-r(\nu-c)}{s} \bigg)  \\ 
&= \sum_{c\in \CA} \prod_{i=1}^{\ell}\sum_{\substack{s_i| g_i^J}} \sum_{\substack{1\leq r_i < s_i \\ (r_i,s_i)=1}} \frac{1}{\varphi(s_i)} \sum_{\nu_i \mod{s_i}} \frac{\rho(\mu_i\nu_i; s_i)}{\varphi(s_i)} e\bigg(\frac{-r_i(\mu_i\nu_i-c)}{s_i} \bigg),
\end{align*}
where we denote $\mu_i :=\prod_{j\not=i}s_j$. Then we can write the individual sum for $s_i$ ($g_i$ coprime to $2$) as
\begin{align*}
 \sum_{\substack{s_i| g_i^J}} \sum_{\substack{1\leq r_i < s_i \\ (r_i,s_i)=1}} \frac{1}{\varphi(s_i)} \sum_{\nu_i \mod{s_i}} &\frac{\rho(\mu_i\nu_i; s_i)}{\varphi(s_i)} e\bigg(\frac{-r_i(\mu_i\nu_i-c)}{s_i} \bigg) \\
 &= 1 + \sum_{\substack{s_i| g_i^J \\ s_i>1}} \sum_{\substack{1\leq r_i < s_i \\ (r_i,s_i)=1}} \frac{1}{\varphi(s_i)} \frac{s_i}{g_i^J} \sum_{\nu_i \mod{g_i^J}} \frac{\rho(\mu_i\nu_i;s_i)}{\varphi(s_i)} e\bigg(\frac{-r_i (\mu_i\nu_i-c)}{s_i} \bigg)  \\
&= 1 + \sum_{\substack{s_i| g_i^J \\ s_i>1}} \sum_{\substack{1\leq r_i < s_i \\ (r_i,s_i)=1}} \frac{1}{\varphi(g_i^J)}\sum_{\nu_i \mod{g_i^J}} \frac{\rho(\mu_i\nu_i;g_i^J)}{\varphi(g_i^J)} e\bigg(\frac{-r_i (\mu_i\nu_i-c)}{s_i} \bigg) \\
&= \sum_{\substack{s_i| g_i^J}} \sum_{\substack{1\leq r_i < s_i \\ (r_i,s_i)=1}} \sum_{\nu_i \mod{g_i^J}} \frac{\rho(\mu_i\nu_i;g_i^J)}{\varphi^2(g_i^J)} e\bigg(\frac{-r_i (\mu_i\nu_i-c)}{s_i} \bigg)
\end{align*}
where to get the first line we used the fact that $\nu \to \nu + ks_i$ for $0\leq k < g_i^J/s_i$ gives the same residue class $\bmod s_i$. For the second line, we use the fact that $\rho(\nu;s_i)/\varphi(s_i)$ is independent of powers of $s_i$ by Lemma \ref{rho_function_properties} and properties of $\varphi(s_i)$. Same idea for $s_i/\varphi(s_i) = g_i^J/\varphi(g_i^J)$. For the last line, we add back the $s_i=1$ to the sum since $(g_i,\mu_i)=1$. The case when one of $g_i$ is power of $2$ is done in the same way, but we take out the cases $s_i=2, 4$ just like $s_i=1$, and then add it back at the end. This can be done because $\rho(\nu;2^a)$ is supported on $\nu\equiv 2\mod 4$.

Now, we glue back the $s_i$ sums as well as the $r_i$ and $\nu_i$. This gives us
$$ \sum_{c\in \CA} \sum_{\substack{s|g^J}} \sum_{\substack{1\leq r < s\\ (r,s)=1}} \sum_{\nu \mod{g^J}} \frac{\rho(\nu;g^J)}{\varphi^2(g^J)} e\bigg( \frac{-r(\nu-c)}{s} \bigg), $$
where we used the fact that $\rho(\mu_i\nu_i;g_i^J)=\rho(\nu;g_i^J)$ since $\nu\equiv \mu_i\nu_i \mod{g_i^J}$, and the multiplicativity of $\rho(\nu;s)$.
Writing $r/s=\ell/g^J$ to get
$$\sum_{\nu \mod{g^J}}\frac{\rho(\nu;g^J)}{\varphi^2(g^J)} \sum_{c\in \CA} \sum_{0\leq \ell <g^J}e\bigg(\frac{-\ell (\nu-c)}{g^J} \bigg) .$$
We note that
$$ \sum_{0\leq \ell <g^J}e\bigg(\frac{-\ell (\nu-c)}{g^J} \bigg) = \begin{cases}
    g^J & \text{if } c\equiv \nu \mod{g^J} \text{ and } \nu=\sum_{i=0}^Ja_ig^i \text{ with } a_i\not=b ,\forall i \\
    0 & \text{otherwise}.
\end{cases} $$
Now, using this identity, we get
$$ \sum_{\substack{\nu \mod{g^J} \\ \nu=\sum_{i=0}^{J-1} a_ig^i , a_i\not=b, \forall i}} g^J \cdot \frac{\rho(\nu;g^J)}{\varphi^2(g^J)} \sum_{\substack{c\in \CA \\ c\equiv \nu \mod{g^J}}} 1 .$$
We can count the sum over missing digits in this case as $(g-1)^{k-J}$, which gives us at the end
\begin{equation}\label{major_arc_main_term}
 \frac{\pi}{4} \cdot \sum_{\substack{s| g^J}}  \sum_{\substack{1\leq r < s\\ (r,s)=1}} \frac{1}{\varphi(s)} \sum_{\nu \mod s} \frac{\rho(\nu; s)}{\varphi(s)} \sum_{c\in \CA}e\bigg(\frac{-r(\nu-c)}{s} \bigg)= \mathfrak{S}(b,g) \cdot \frac{\pi}{4} \#\CA(X) ,
\end{equation}
where
\begin{align*}
\mathfrak{S}(b,g) = \frac{g^J}{(g-1)^J} \cdot \sum_{\substack{\nu \mod{g^J} \\ \nu=\sum_{i=0}^{J-1} a_ig^i , a_i\not=b,\forall i}} \frac{\rho(\nu;g^J)}{\varphi^2(g^J)} &= \frac{1}{(g-1)^J} \cdot \sum_{\substack{\nu \mod{g^J} \\ \nu=\sum_{i=0}^{J-1} a_ig^i , a_i\not=b,\forall i}} \frac{g \cdot \rho(\nu;g)}{\varphi^2(g)} \\
&= \frac{g}{g-1} \bigg( 1 - \frac{\rho(b;g)}{\varphi^2(g)} \bigg),
\end{align*}
where we used the fact that $g^J\rho(\nu;g^J)/\varphi^2(g^J)$ is independent of the power of $g$ to get the first line, and noticing that $\rho(\nu;g)$ only depends on the last digit of $\nu$ in base $g$, hence if $\nu=\sum_{i=0}^{J-1} a_i g^i$, summing over $0\leq a_i<g$ for $0<i\leq J-1$, we get a $(g-1)^{J-1}$ factor. The remaining sum over $0\leq a_0<g$ gives the factor with $\rho(b;g)$.

\subsection{Minor Arcs}
We wish to estimate the sum
$$ \frac{1}{X} \sum_{a\in \m} \hat{1}_{\CA}\bigg(\frac{a}{X}\bigg) S\bigg(\frac{-a}{X}\bigg)  . $$
By the Dirichlet approximation theorem, if $a\in \m$, write $\frac{a}{X} = \frac{r}{s} + \frac{\eta}{X}$ for $s\leq X^{3/4}$ with $(\log X)^B \leq s|\eta| \leq \sqrt{X}$. Also, note that $s^2 |\eta|\leq X$ by Dirichlet's approximation theorem.
Next, we put $s$ and $\eta$ into dyadic intervals. Hence, we have $(\log X)^B \leq DS < \sqrt{X}$ with $S=2^i, D=2^j$, and we are studying
$$ \frac{1}{X} \sum_{\substack{S<s\leq 2S}}\sum_{\substack{1 \leq r < s  \\ (r,s)=1}} \sum_{\substack{D<|\eta| \leq 2D \\ Xr/s+\eta\in \Z}} \hat{1}_{\CA}\bigg(\frac{r}{s}+ \frac{\eta}{X}\bigg) S\bigg(\frac{-r}{s} + \frac{-\eta}{X}\bigg). $$
Now, for $D\ll 1$ or $S\gg X^{1/4}$, we use Lemma \ref{exponential_sum_of_r_2}[(ii)] on $S(\theta)$ and Lemma \ref{hybrid_bound_for_fourier_transform} on the remainder average of the Fourier transform $\hat{1}_{\CA}(\theta)$, \textit{i.e.}
$$ \sum_{\substack{S<s\leq 2S}}\sum_{\substack{1 \leq r < s  \\ (r,s)=1}} \sum_{\substack{D<|\eta| \leq 2D \\ Xr/s+\eta\in \Z}} \bigg| \hat{1}_{\CA}\bigg( \frac{r}{s} + \frac{\eta}{X}\bigg)\bigg| \ll \#\CA(X) (DS^2)^{\alpha_g}, $$
to get
\begin{align*}
&\ll \#\CA(X) (\log X) \bigg( \frac{(DS^2)^{192\alpha_g}}{S} + \frac{(DS^2)^{192\alpha_g}}{X^{1/3}} + \frac{(DS^2)^{192\alpha_g}S}{X^2} \bigg)^{1/192}.
\end{align*}
For $D\ll 1$, we use $(\log X)^B\ll S\ll \sqrt{X}$ to get
\begin{align}
&\ll \#\CA(X) (\log X) \bigg( \frac{1}{S^{1-384\alpha_g}} + \frac{S^{384\alpha_g}}{X^{1/3}} + \frac{S^{384\alpha_g}}{X^{3/2}} \bigg)^{1/192} \nonumber \\
&\ll \frac{\#\CA(X)}{(\log X)^{-1+(1-384\alpha_g)B/384}}, \label{good_bound}
\end{align}
since $\alpha_g \to 0$ as $g\to \infty$, we can further bound this by $\ll \#\CA(X) / (\log X)^{(B-385)/384}$ for $g$ large enough. In the case of $S\gg X^{1/4}$, we use $DS^2\ll X$ to get
\begin{align}
&\ll \#\CA(X) (\log X) \bigg( \frac{(DS^2)^{192\alpha_g}}{X^{1/4}} + \frac{(DS^2)^{192\alpha_g}}{X^{1/3}} + \frac{(DS^2)^{192\alpha_g}}{X^{3/2}} \bigg)^{1/192} \notag \\
&\ll \frac{\#\CA(X) (\log X)}{X^{(1/4-192\alpha_g)/192}},\label{worst_bound}
\end{align}
and again since $\alpha_g\to 0 $ as $g\to\infty$, we get that this contributes $\ll \#\CA(X) X^{-1/1000}$.

For the other cases $D\gg 1$ and $S\ll X^{1/4}$, we use Lemma \ref{exponential_sum_of_r_2}[(i)] instead (with $\beta=\eta/X$) to get
\begin{align*}
\frac{1}{X} \sum_{\substack{S<s\leq 2S}}\sum_{\substack{1 \leq r < s  \\ (r,s)=1}} &\sum_{\substack{D<|\eta| \leq 2D \\ Xr/s+\eta\in \Z}} \hat{1}_{\CA}\bigg(\frac{r}{s}+ \frac{\eta}{X}\bigg) S\bigg(\frac{-r}{s} + \frac{-\eta}{X}\bigg) \\
&\ll \frac{\#\CA(X)}{X} \cdot (DS^2)^{\alpha_g} \bigg(X^{15/16+\varepsilon} + \frac{X}{S^{1/2-\varepsilon}D^{1/2}} \bigg) (\log X)^c.
\end{align*}
We can bound $D^{-1/2} \ll D^{-1/2+\varepsilon}$ and $(DS^2)^{\alpha_g} \leq (DS)^{2\alpha_g}$. These together, along with $DS^2 \ll X$, gives us
\begin{align}
\ll \#\CA(X) \bigg( X^{-1/16+\alpha_g} + (DS)^{-1/2+\varepsilon + 2\alpha_g} \bigg) (\log X)^c \label{okay_bound}
\end{align}
Now, as before, since $\alpha_g \to 0$ as $g\to\infty$, we have for sufficiently large $g$ that $-1/4 +\varepsilon +2\alpha_g \leq -1/3$ and $-1/16 + \alpha_g/2 \leq -1/17$. Since $DS\gg (\log X)^B$, we get that we have $\ll \#\CA(X) /(\log X)^{(B-3c)/3}$ in this case. 

Thus, combining \eqref{good_bound}, \eqref{worst_bound}, and \eqref{okay_bound}, the minor arc contribution in total can be bounded by $\ll \#\CA(X)/(\log X)^{\min\{(B-3c)/3,(B-1153)/384\}}$.

Finally, combining the error term from the minor arcs with the major arc estimates \eqref{major_arc_error_case_1}, \eqref{major_arc_error_case_2}, and the main term \eqref{major_arc_main_term}, we obtain, for any sufficiently large $B>0$,
$$ \sum_{n\in \CA(X)} r_2(n) = \mathfrak{S}(b,g)\cdot\frac{\pi}{4}\#\CA(X) + O_B \left(\frac{\#\CA(X)}{(\log X)^B}\right), $$
which completes the proof.

\section{Off-Diagonal Solutions}\label{off_diagonal_section}
To understand Theorem \ref{off-diagonal_result}, we first get rid of the following quantity
$$ \sum_{\substack{n\in \CA(X) \\ n\leq X^{3/4}}} (r_*^2(n) - 2r_*(n)) \ll \sum_{n\leq X^{3/4}} \tau^2(n) \ll_{\varepsilon} X^{3/4+\varepsilon}, $$
since $r_*(n)\leq \tau(n)$. Moreover, since our sum can be written in the following form
$$ \sum_{\substack{n\in \CA(X) \\ n>X^{3/4}}} (r_*^2(n)-2r_*(n)) = \sum_{\substack{p_1^2+q_1^2 = p_2^2+q_2^2 \in \CA(X) \\ p_1^2+q_1^2 > X^{3/4} \\ \{p_1,q_1\}\not=\{p_2,q_2\}}} 1. $$
If one of $p_i$ or $q_i$ divides $g$, then we have the bound
$$ \leq \sum_{\substack{p_1^2+q_1^2 \leq X \\ p_1|3g }} r_*(p_1^2+q_1^2) \ll_{\varepsilon, g} X^{1/2+\varepsilon}  ,$$
since $r_*(n) \ll_{\varepsilon} n^{\varepsilon}$. Next, we follow an idea of the author \cite{MR4756119}, which we give here. It starts with the following observation. We want to understand the following sum,
$$\sum_{\substack{p_1^2+q_1^2 = p_2^2+q_2^2 \in \CA(X) \\ p_1^2+q_1^2 > X^{3/4} \\ (p_1p_2q_1q_2,3g)=1 \\ \{p_1,q_1\}\not=\{p_2,q_2\}}} 1. $$
Hence, we can split $p_1^2+q_1^2 = p_2^2+q_2^2$ in $\Z[i]$ to get
\begin{align*}
p_1+iq_1 = (a+ib)(c+id) \\
p_2+iq_2 = (a+ib)(c-id),
\end{align*}
up to some unit factors, which we can disgard without loss of generality by just focusing on this case.
Hence, using this factorization, we can get an upper bound
$$ \sum_{n\in \CA(X)} (r_*^2(n)-2r_*(n)) \ll_{\varepsilon}  \sum_{\substack{(a^2+b^2)(c^2+d^2) \in \CA(X) \\ (a^2+b^2)(c^2+d^2)> X^{3/4} \\ (ac-bd), (ac+bd), (ad+bc), (ad-bc) \in \CP_g}} 1  + X^{3/4+\varepsilon}, $$
where $\CP_g$ is the set of primes coprime to $3g$.
We can further assume without loss of generality that $(a^2+b^2)\leq (c^2+d^2)$. Thus, we see that the sum on the
right-hand side of the above expression is
$$ \ll \sum_{a^2+b^2 \leq \sqrt{X}} \sum_{\substack{X^{3/4}/(a^2+b^2) < c^2+d^2 \leq X/(a^2+b^2) \\ (a^2+b^2)(c^2+d^2) \in \CA \\ (ac-bd), (ac+bd), (ad+bc), (ad-bc) \in \CP_g}} 1 .$$
Next, we use Lemma \ref{sieve_lemma} to get rid of the primality condition in the inner sum and bound this by
\begin{equation}\label{main_eq}
 \sum_{a^2+b^2 \leq \sqrt{X}}  \sum_{\substack{h\leq H \\ h|P(z)}} \lambda_h \sum_{\substack{v_1,v_2 \mod h \\ (av_1 - bv_2)(av_1 + bv_2)(av_2 + b v_1)(av_2 - bv_1) \equiv 0 \mod h}} \sum_{\substack{X^{3/4}/(a^2+b^2) < c^2+d^2 \leq X/(a^2+b^2) \\ (a^2+b^2)(c^2+d^2) \in \CA  \\  (c,d)\equiv (v_1,v_2) \mod h }} 1 , 
\end{equation}
where $z=\exp(\frac{c_g^{1/2}}{3500000}(\log X)^{1/2})$, and $H=z^{70}$ with $c_g$ is defined as in Lemma \ref{Linfty_bound_for_fourier_transform}. Recall that we also have $(h,g)=1$ since $P(z)=\prod_{\substack{p<z \\ p\nmid g}}p$. 
We denote $M:= a^2+b^2$, and
\begin{equation*}
T(X, h ; v_1, v_2)  :=  \sum_{\substack{X^{3/4}/M < c^2+d^2 \leq X/M \\ M(c^2+d^2) \in \CA  \\  (c,d)\equiv (v_1,v_2) \mod h }} 1.
\end{equation*}
By the Fourier inversion, we get
$$T(X, h ; v_1, v_2) = \frac{1}{X} \sum_{\tilde{a}\leq X} \hat{1}_{\CA} \bigg( \frac{\tilde{a}}{X}\bigg) \tilde{S}\bigg( \frac{-\tilde{a}}{X} ; v_1,v_2 \bigg) ,$$
where
$$ \tilde{S}(\theta ; u,v) := \sum_{\substack{X^{3/4}/M < c^2+d^2\leq X/M \\  (c,d)\equiv (u,v)\mod h}} e(M(c^2+d^2) \theta) . $$
Now, we split into major and minor arcs. We can define for $\eta_g = c_g^{1/2}/1000$,
$$ \M_{r,s} := \bigg\{ a \in [1,X] \cap \Z : \bigg|\frac{a}{X} - \frac{r}{s} \bigg| \leq \frac{\exp(\eta_g\cdot (\log X)^{1/2})}{X} \bigg\} . $$
Then, major arcs are defined as
$$ \M := \bigcup_{s\leq \exp(\eta_g\cdot (\log X)^{1/2})}\bigcup_{\substack{ 1 \leq r < s\\ (r,s)=1}} \M_{r,s}, $$
and minor arcs are
$$ \m := ([1,X] \cap \Z)\setminus \M. $$
Then we have
\begin{align*}
T(X, h ; v_1, v_2) &= \frac{1}{X} \sum_{s\leq  \exp(\eta_g\cdot (\log X)^{1/2})} \sum_{\substack{1 \leq r < s \\ (r,s)=1}} \sum_{\substack{|\eta| \leq \exp(\eta_g\cdot (\log X)^{1/2}) \\ Xr/s+\eta \in \Z}} \hat{1}_{\CA} \bigg( \frac{\tilde{r}}{s} + \frac{\eta}{X}\bigg) \tilde{S}\bigg( \frac{\tilde{-r}}{s} + \frac{-\eta}{X} ; v_1,v_2 \bigg) \\
&+ \frac{1}{X} \sum_{\tilde{a} \in \m} \hat{1}_{\CA} \bigg( \frac{\tilde{a}}{X}\bigg) \tilde{S}\bigg( \frac{-\tilde{a}}{X}\bigg)
\end{align*}

\subsection{Major Arcs}
For this part, we focus on getting the main term for the inner sum in \eqref{main_eq}. Later on, we will sum it over $h$, and $a^2+b^2$ to get the actual main term.

For the major arcs, we have $\frac{\tilde{a}}{X} = \frac{r}{s} + \frac{\eta}{X}$ with $|\eta| \ll \exp(\eta_g\cdot (\log X)^{1/2}) $. We again split into 3 cases:

\noindent \textit{\textbf{Case 1. $\exists p|s$ but $p\nmid g$:}} 
In this case by Lemma \ref{Linfty_bound_for_fourier_transform}, we get 
\begin{align*}
\frac{1}{X} \sum_{\substack{s\leq \exp(\eta_g\cdot (\log X)^{1/2}) \\ \exists p|s, p\nmid g}} \sum_{\substack{1\leq r < s \\ (r,s)=1}} \sum_{\substack{|\eta|\leq \exp(\eta_g\cdot (\log X)^{1/2}) \\ Xr/s + \eta\in \Z}}&\hat{1}_{\CA}\bigg(\frac{r}{s}+ \frac{\eta}{X}\bigg) \tilde{S}\bigg(\frac{-r}{s} + \frac{-\eta}{X}; v_1,v_2\bigg) \\
&\ll \frac{\#\CA(X)}{Mh^2} \sum_{\substack{s\leq \exp(\eta_g\cdot (\log X)^{1/2}) \\ \exists p|s, p\nmid g \\ (r,s)=1}}\sum_{\substack{|\eta|\leq \exp(\eta_g\cdot (\log X)^{1/2}) \\ Xr/s + \eta\in \Z}} X^{-c_g/\log s} \\
&\ll _{\varepsilon}\frac{\#\CA(X)}{Mh^2}\cdot \exp( - 100 c_g^{1/2} (\log X)^{1/2}) ,
\end{align*}
by trivially bounding $\tilde{S}(-r/s-\eta/X;v_1,v_2) \ll X/h^2M$ since $h \leq H$ and $H$ is very small.

\noindent \textit{\textbf{Case 2. $\eta\not=0$:}} Now, we need to understand
\begin{align*}
\tilde{S}\bigg( \frac{-a}{X} ; v_1,v_2 \bigg) &= \sum_{\ell_1,\ell_2 \mod s} e\bigg( \frac{-rM(\ell_1^2+\ell_2^2)}{s} \bigg) \sum_{\substack{ X^{3/4}/M < c^2+d^2 \leq X/M \\ (c,d)\equiv (v_1,v_2)\mod h \\ (c,d)\equiv (\ell_1,\ell_2) \mod s}} e\bigg( \frac{-\eta M (c^2+d^2)}{X} \bigg) \\
&= \sum_{\ell_1,\ell_2 \mod s} e\bigg( \frac{-rM(\ell_1^2+\ell_2^2)}{s} \bigg) \sum_{\substack{c^2+d^2 \leq X/M \\ (c,d)\equiv (C,D)\mod{[h,s]}}} e\bigg( \frac{-\eta M (c^2+d^2)}{X} \bigg) + O\bigg( \frac{X^{3/4}}{M}  \bigg) ,
\end{align*}
where $C$ and $D$ are integers obtained by the Chinese Remainder Theorem, completing the sum over $c^2+d^2$ and using Lemma \ref{sum_of_squares_in_short_int_and_aps} with the fact that $s \leq \exp(\eta_g \cdot (\log X)^{1/2})$ is very small. Next, we split the $c^2+d^2$ sum into short intervals to get (for $Y:=\lfloor (X/M)^{1/3}\rfloor$)
\begin{align*}
&\sum_{0\leq m< Y} \sum_{\substack{mX/YM< c^2+d^2 \leq (m+1) X/YM \\ (c,d)\equiv (C,D)\mod{[h,s]}}} e\bigg( \frac{-\eta M(c^2+d^2)}{X} \bigg)  \\
&= \sum_{0\leq m< Y}e\bigg( \frac{-\eta m}{Y} \bigg) \sum_{\substack{mX/YM< c^2+d^2 \leq (m+1)X/YM \\ (c,d)\equiv (C,D)\mod{[h,s]}}} 1 + O\bigg( \frac{|\eta|\cdot  X^{1/3}}{M^{1/3}} \bigg) \\
&= \frac{1}{[h,s]^2} \frac{X}{YM} \sum_{0\leq m< Y}e\bigg( \frac{-\eta m}{Y} \bigg) + O\bigg( \frac{X^{5/6}}{M^{5/6}} \bigg) \ll  \frac{X^{5/6}}{M^{5/6}},
\end{align*}
by Lemma \ref{sum_of_squares_in_short_int_and_aps}, and since the sum over $m$ is 0. Hence, bounding the Fourier transform $\hat{1}_{\CA}(\theta)$ trivially and since $X^{3/4}/M \leq X^{5/6}/M^{5/6}$, we get
$$ \frac{1}{X} \sum_{\substack{s\leq \exp(\eta_g\cdot (\log X)^{1/2}) }} \sum_{\substack{1\leq r < s \\ (r,s)=1}} \sum_{\substack{0<|\eta|\leq \exp(\eta_g\cdot (\log X)^{1/2}) \\ Xr/s + \eta\in \Z}}\hat{1}_{\CA}\bigg(\frac{r}{s}+ \frac{\eta}{X}\bigg) \tilde{S}\bigg(\frac{-r}{s} + \frac{-\eta}{X}; v_1,v_2\bigg) \ll \frac{\#\CA(X)}{X^{1/6-\varepsilon}M^{5/6}} .$$

\noindent \textit{\textbf{Case 3. $\eta=0$ and $p|s \implies p|g$:}} This part will give us the main term. We have similar to the Section \ref{average_section},
$$  \frac{1}{X} \sum_{\substack{s\leq \exp(\eta_g\cdot (\log X)^{1/2}) \\ p|s\implies p|g }}  \sum_{\substack{1\leq r < s \\ (r,s)=1}}\hat{1}_{\CA}\bigg(\frac{r}{s}\bigg) \sum_{\ell_1,\ell_2 \mod s } e\bigg( \frac{-rM(\ell_1^2+\ell_2^2)}{s} \bigg) \sum_{\substack{X^{3/4}/M < c^2+d^2\leq X/M \\ (c,d)\equiv (C,D)\mod{hs}}} 1, $$
since $(h,g)=1$ and so $(h,s)=1$. By Lemma \ref{sum_of_squares_in_short_int_and_aps}, we can get that this is
$$ \frac{\pi}{4}\frac{1}{Mh^2} \sum_{\substack{s\leq \exp(\eta_g\cdot (\log X)^{1/2}) \\ p|s\implies p|g}}  \sum_{\substack{1\leq r < s \\ (r,s)=1}}\hat{1}_{\CA}\bigg(\frac{r}{s}\bigg) \sum_{\ell_1,\ell_2 \mod s}\frac{1}{s^2} e\bigg( \frac{-rM(\ell_1^2+\ell_2^2)}{s} \bigg) + O_{\varepsilon}\bigg( \frac{\#\CA(X)}{X^{1/2-\varepsilon}\sqrt{M}} + \frac{\#\CA(X)}{X^{1/4-\varepsilon} M} \bigg)  . $$
 Now, we can write $\nu=\ell_1^2+\ell_2^2$ to get
$$ \frac{\pi}{4}\frac{1}{Mh^2} \sum_{\substack{s\leq \exp(\eta_g\cdot (\log X)^{1/2}) \\ p|s\implies p|g}} \sum_{\substack{1\leq r < s \\ (r,s)=1}}\hat{1}_{\CA}\bigg(\frac{r}{s}\bigg) \sum_{\nu \mod s}\frac{r(\nu ;s)}{s^2} e\bigg( \frac{-rM\nu}{s} \bigg) + O_{\varepsilon}\bigg( \frac{\#\CA(X)}{X^{1/4-\varepsilon}M} \bigg)  , $$
where
$$ r(\nu ;s) := \sum_{\ell_1, \ell_2 \mod s} 1_{\nu\equiv \ell_1^2+\ell_2^2 \mod s}. $$
Now, since $p|s\implies p|g$, we just need to study for any choice of $J\ll_g (\log X)^{1/2}$, the following\footnote{The bound on $J$ comes from the fact that $g^J\leq \exp(\eta_g \cdot (\log X)^{1/2})$.}
\begin{equation}\label{getting_rid_of_s}
 \sum_{\substack{s|g^J }} \sum_{\substack{1\leq r < s \\ (r,s)=1}}\hat{1}_{\CA}\bigg(\frac{r}{s}\bigg) \sum_{\nu \mod s}\frac{r(\nu ;s)}{s^2} e\bigg( \frac{-rM\nu}{s} \bigg).
\end{equation}

Next, we note that for a given $(a,b)\in (\Z/s\Z)^2$ with $a^2+b^2\equiv \nu \mod s$, we have $(g^J/s)^2$ distinct choices for $(k_1,k_2)$ such that $(a+k_1s,b+k_2s)\in (\Z/g^J\Z)^2$ with $(a+k_1s)^2+(b+k_2s)^2 \equiv \mu \mod{g^J}$ for some $\mu\equiv \nu \mod s$. Hence, we have
$$ r(\nu;s)= \frac{s^2}{g^{2J}} \sum_{\substack{\mu\mod{g^J} \\ \mu \equiv \nu \mod s}} r(\mu;g^J) .$$
Putting this in \eqref{getting_rid_of_s} and changing the order of summations, we get
$$ \frac{\pi}{4}\frac{1}{Mh^2} \sum_{\mu \mod{g^J}} \frac{r(\mu;g^J)}{g^{2J}} \sum_{\substack{s|g^J}} \sum_{\substack{1\leq r < s \\ (r,s)=1}}\hat{1}_{\CA}\bigg(\frac{r}{s}\bigg) e\bigg( \frac{-rM\mu}{s} \bigg) . $$
Opening up the Fourier transform $\hat{1}_{\CA}$ and changing the order of summation we get a sum over $s$ of $e(r(c - M\mu)/s)$ where $c\in\CA(X)$. Writing $r/s=\ell/g^J$, we get the sum
$$ \sum_{0\leq \ell < g^J} e\bigg(\frac{r(c - M\mu)}{s}\bigg) = \begin{cases}
    g^J & \text{if } c\equiv M\mu \mod{g^J} \text{ and } M\mu=\sum_{i=0}^{k-1}a_ig^i \text{ with } a_i\not=b ,\forall i< J \\
    0 & \text{otherwise}.
\end{cases}  $$
Putting this in, we get
$$ \eqref{getting_rid_of_s} = \frac{\pi}{4} \frac{1}{Mh^2} \frac{1}{g^J}\sum_{\substack{\mu \mod{g^J} \\ M\mu \equiv \sum_{i=0}^{J-1}a_ig^i \mod{g^J} \text{ with } a_i\not=b, \forall i}} r(\mu ; g^J) \sum_{\substack{c\in \CA(X) \\ c\equiv \mu M \mod{g^J}}} 1 . $$
Similar to Section \ref{average_section}, the inner sum over $c$ is $(g-1)^{k-J}$. Thus, we get that this is
$$  \frac{\pi}{4} \frac{\#\CA(X)}{Mh^2} \frac{1}{g^J(g-1)^J}\sum_{\substack{\mu \mod{g^J} \\ M\mu = \sum_{i=0}^{k-1}a_ig^i \text{ with } a_i\not=b, \forall i<J}} r(\mu ; g^J) .$$

Thus, the major arc contribution in total gives us, for any $J \ll_g (\log X)^{1/2}$,
\begin{align}
\sum_{\substack{\tilde{a} \in \M}} \hat{1}_{\CA}\bigg(\frac{\tilde{a}}{X}\bigg) \tilde{S}\bigg( \frac{-\tilde{a}}{X} ; v_1 , v_2 \bigg) = \frac{\pi}{4} \frac{\#\CA(X)}{Mh^2} \frac{1}{g^J(g-1)^J}&\sum_{\substack{\mu \mod{g^J} \\ M\mu \equiv \sum_{i=0}^{J-1}a_ig^i \mod{g^J} \text{ with } a_i\not=b, \forall i}} r(\mu ; g^J) \label{major_arc_contr}  \\
&+ O\bigg( \frac{\#\CA(X)}{Mh^2} \exp\bigg( - 100 c_g^{1/2} (\log X)^{1/2}  \bigg) \bigg), \notag
\end{align}
since $M\leq \sqrt{X}$ and summing over $M=a^2+b^2$ would give us that this is the worst possible error term.

\subsection{Minor Arcs}
We want to bound the following sum on average 
$$ \frac{1}{X} \sum_{\tilde{a}\in \m} \hat{1}_{\CA} \bigg( \frac{\tilde{a}}{X}\bigg) \tilde{S}\bigg( \frac{-\tilde{a}}{X} ; v_1,v_2 \bigg) .$$
Similar to the first part of the paper, by the Dirichlet approximation theorem, we can write $\frac{\tilde{a}}{X} = \frac{r}{s} + \frac{\eta}{X}$ for $s\leq \sqrt{X}$ with $\exp(\eta_g\cdot (\log X)^{1/2}) \leq |\eta|s \leq \sqrt{X}$ and $|\eta| s^2 \leq X$.
We start by putting $s$ and $\eta$ into dyadic intervals. This gives us
\begin{align*}
\sum_{a^2+b^2 \leq \sqrt{X}}  \sum_{h\leq H} \lambda_h &\sum_{\substack{v_1,v_2 \mod h \\ (av_1 - bv_2)(av_1 + bv_2)(av_2 + b v_1)(av_2 - bv_1) \equiv 0 \mod h}} \\
&\times \frac{1}{X} \sum_{\substack{S<s\leq 2S}} \sum_{\substack{1 \leq r < s  \\ (r,s)=1}} \sum_{\substack{D<|\eta|\leq 2D \\ Xr/s + \eta\in \Z}} \hat{1}_{\CA}\bigg( \frac{r}{s} + \frac{\eta}{X} \bigg) \tilde{S}\bigg(  \frac{-r}{s} + \frac{-\eta}{X} ; v_1,v_2 \bigg) .
\end{align*}
Then we put $a^2+b^2$ into dyadic intervals and change the order of summations to put it inside by forcing $a$ and $b$ into arithmetic progressions, \textit{i.e.} we have
\begin{align}
\leq  \sum_{h\leq H}&\bigg|\sum_{\substack{u_1,u_2,v_1,v_2 \mod h \\ (u_1v_1 - u_2v_2)(u_1v_1 + u_2v_2)(u_1v_2 + u_2v_1)(u_1v_2 - u_2v_1)  \equiv 0 \mod h}} \notag \\
&\times\frac{1}{X} \sum_{\substack{S<s\leq 2S}}\sum_{\substack{1 \leq r < s  \\ (r,s)=1}} \sum_{\substack{D<|\eta|\leq 2D \\ Xr/s + \eta\in \Z}} \hat{1}_{\CA}\bigg( \frac{r}{s} + \frac{\eta}{X} \bigg) \sum_{N=2^i \ll \sqrt{x}} \sum_{\substack{N <a^2+b^2 \leq 2N \\ (a,b)\equiv (u_1,u_2)\mod h }}S\bigg(  \frac{-r}{s} + \frac{-\eta}{X} ; v_1,v_2 \bigg) \bigg| . \label{minor_arc_part}
\end{align}
Now, we focus on the \eqref{minor_arc_part} part only. We open up $S(\theta;u,v)$, and put $(c,d)$ in dyadic intervals $(M_1,2M_1]\times (M_2,2M_2]$ for some real numbers $M_1,M_2$ with $\#\{M_1, M_2\} \ll (\log X)^2$ to get
\begin{align*}
\sum_{\substack{N <a^2+b^2 \leq 2N \\ (a,b)\equiv (u_1,u_2)\mod h }}\tilde{S}\bigg(  \frac{-r}{s} +& \frac{-\eta}{X} ; v_1,v_2 \bigg) = \sum_{\substack{M_1 , M_2 \\ M_1^2+M_2^2 \ll X/N}} \sum_{\substack{N <a^2+b^2 \leq 2N \\ (a,b)\equiv (u_1,u_2)\mod h }} \\ 
&\times\sum_{\substack{M_1< c \leq 2M_1 \\ M_2 < d \leq 2M_2 \\ X^{3/4}/(a^2+b^2)<c^2+d^2 \leq X/(a^2+b^2) \\ (c,d)\equiv (v_1,v_2) \mod h}} e\bigg(  \bigg( \frac{-r}{s} + \frac{-\eta}{X} \bigg)(a^2+b^2)(c^2+d^2)\bigg). 
\end{align*}
Now, we can put $n=a^2+b^2$, and this tells us the inner sum is of the form
$$ \sum_{N<n\leq 2N} a_n \sum_{\substack{M_1< c \leq 2M_1 \\ M_2<d\leq 2M_2 \\ X^{3/4}/(a^2+b^2)<c^2+d^2 \leq X/(a^2+b^2)  \\ (c,d)\equiv (v_1,v_2)\mod h}} e\bigg(  \bigg( \frac{-r}{s} + \frac{-\eta}{X} \bigg) (c^2+d^2)n \bigg) ,$$
which we can bound by Lemma \ref{quadratic_exponential} with $\delta=1/8$ since $c^2+d^2 \gg X^{3/4}/(a^2+b^2) \gg X^{1/4}$. This, alongside Lemma \ref{hybrid_bound_for_fourier_transform}, gives us for $D\gg 1$,
\begin{align*}
\eqref{minor_arc_part} &\ll \frac{h^{1/2}}{X} \#\CA(X) (S^2 D)^{\alpha_g} \bigg( X^{31/32}  + \frac{X}{(SD)^{1/4}} \bigg) (\log X)^6 \\
&\ll h^{1/2} \#\CA(X) \bigg( \frac{(S^2D)^{\alpha_g}}{X^{1/32}} + \frac{(S^2D)^{\alpha_g}}{(SD)^{1/4}} \bigg) (\log X)^6 \\
&\ll h^{1/2} \#\CA(X) ( X^{-1/32 + \alpha_g} + (S D)^{-1/4+2\alpha_g})(\log X)^6,
\end{align*}
since $S^2D\ll X$ and $S^2D\leq (S D)^2$. Next, since $\alpha_g\to 0$ as $g\to\infty$, we get that $X^{-1/32 +\alpha_g} \ll X^{-1/100}$ and using $SD\gg \exp(\eta_g\cdot (\log X)^{1/2})$, we have $(SD)^{-1/4+2\alpha_g} \ll \exp(-\frac{\eta_g}{5} (\log X)^{1/2})$. 

In the case $D\ll 1$, we can use the second bound from Lemma \ref{quadratic_exponential}. In this case, we get
\begin{align*}
\eqref{minor_arc_part} &\ll \frac{h^{1/2}}{X} \#\CA(X) S^{2\alpha_g} \bigg( X^{31/32} + \frac{X}{S^{1/4}}\bigg) (\log X)^6 \\
&\ll h^{1/2} \#\CA(X) (S^{2\alpha_g} X^{-1/32} + S^{-1/4+2\alpha_g})(\log X)^6.
\end{align*}
As before, since $\alpha_g\to 0$ as $g\to\infty$, and $S\gg \exp(\eta_g\cdot (\log X)^{1/2})$, we get that this is 
$$\ll h^{1/2} \#\CA(X) \exp(-\frac{\eta_g}{5} (\log X)^{1/2}) ,$$ 
just like the case $D\gg 1$. Therefore, putting them together, we got
$$ \eqref{minor_arc_part} \ll h^{1/2} \#\CA(X) \exp\bigg(- \frac{\eta_g}{5} (\log X)^{1/2}\bigg) .  $$
Summing subsequently over $u_1,u_2,v_1,v_2 \mod h$ and $h\leq H$ gives a bound of
\begin{align}
\sum_{a^2+b^2\leq \sqrt{X}}\sum_{h\leq H} \sum_{\substack{v_1,v_2 \mod h \\ (av_1 - bv_2)(av_1 + bv_2)(av_2 + b v_1)(av_2 - bv_1) \equiv 0 \mod h}}&\sum_{\substack{\tilde{a}\in \m}} \hat{1}_{\CA} \bigg( \frac{\tilde{a}}{X} \bigg) \tilde{S}\bigg( \frac{-\tilde{a}}{X} ; v_1,v_2 \bigg) \label{minor_arc_contr} \\
&\ll \#\CA(X) H^{11/2} \exp\bigg( -\frac{\eta_g}{5} (\log X)^{1/2} \bigg) . \notag
\end{align}
Before going to the next part, we note that this error is worse than the error coming from the major arc contributions.

\subsection{Putting Everything Together}
In this section, we put together the major arc and the minor arc contributions to get the final bound.

We start by noting that
$$ \rho(h; a,b) := \sum_{\substack{v_1,v_2 \mod h \\ (av_1 - bv_2)(av_1 + bv_2)(av_2 + b v_1)(av_2 - bv_1) \equiv 0 \mod h}} 1  ,$$
is a multiplicative function by the Chinese Remainder Theorem. Moreover, we have at a prime $p\not=3$
$$  \rho(p;a,b) = \begin{cases}
    p & \text{if } p=2 \text{ or } p|ab \\
    4p-3 & \text{otherwise}.
\end{cases}  $$
Thus, we get by Lemma \ref{sieve_lemma},
$$ \sum_{h\leq H} \frac{\lambda_h \cdot \rho(h;a,b)}{h^2} \ll \prod_{\substack{p\leq H \\ p\nmid 3g}} \bigg( 1 - \frac{\rho(p;a,b)}{p^2}\bigg) = \prod_{\substack{p\leq H \\ p|ab \\ p\nmid 3g}} \bigg( 1 - \frac{1}{p} \bigg) \prod_{\substack{3<p\leq H \\ p\nmid abg}} \bigg( 1 - \frac{4}{p} + \frac{3}{p^2} \bigg) . $$
We can bound the product over $p|ab$ by $\leq 1$. For the product over $p\nmid abg$, we complete the product to get
$$ \sum_{h\leq H} \frac{\lambda_h \cdot \rho(h;a,b)}{h^2} \ll \prod_{3<p\leq H} \bigg( 1 - \frac{4}{p} + \frac{3}{p^2} \bigg) \prod_{\substack{p>3 \\ p|abg}} \bigg( 1 - \frac{4}{p} + \frac{3}{p^2} \bigg)^{-1} . $$
Then we have that these products are
$$ \ll \prod_{3<p\leq H} \bigg( 1 - \frac{4}{p} \bigg) \exp\bigg(\sum_{p|abg} \frac{4}{p} \bigg). $$
Now, the exponential term can be shown to be $\ll (\log\log X)^4$ since $abg\ll_g X$ by splitting the sum over $p$ into ranges $p>\log X$ and $p\leq \log X$. For the first product, it is easy to see that it is $\ll 1/(\log H)^4$. These together give us
\begin{equation}\label{sieve_upper_bound}
 \sum_{h\leq H} \frac{\lambda_h \cdot \rho(h;a,b)}{h^2} \ll \frac{(\log\log X)^4}{(\log H)^4} .
\end{equation}
We now can use this to prove our result. We put together the major arc \eqref{major_arc_contr} and the minor arc \eqref{minor_arc_contr} contributions together to get
\begin{align*}
 \eqref{main_eq} = \frac{\pi}{4} \#\CA(X) \sum_{a^2+b^2 \leq \sqrt{X}} \frac{1}{a^2+b^2} &\sum_{\substack{\mu \mod{g^J} \\ (a^2+b^2)\mu \equiv \sum_{i=0}^{J-1}a_ig^i \mod{g^J} \text{ with } a_i\not=b, \forall i}} \frac{r(\mu ; g^J)}{g^J(g-1)^J} \\
&\sum_{h\leq H} \frac{\lambda_h \cdot \rho(h;a,b)}{h^2}  + O\bigg( \#\CA(X) H^{11/2}\exp\bigg( - \frac{c_g^{1/2}}{5} (\log X)^{1/2}  \bigg) \bigg).
\end{align*}
Using \eqref{sieve_upper_bound}, we can bound the sum over $h$. Noting that $H=\exp( \frac{c_g^{1/2}}{50000}(\log X)^{1/2})$, and putting $m=a^2+b^2$ we get
\begin{align*}
\ll \frac{\#\CA(X)(\log\log X)^4}{(\log X)^2} \sum_{m\leq \sqrt{X}} \frac{r_0(m)}{m} \sum_{\substack{\mu \mod{g^J} \\ m\mu \equiv \sum_{i=0}^{J-1}a_ig^i \mod{g^J} \text{ with } a_i\not=b, \forall i}}& \frac{r(\mu ; g^J)}{g^J(g-1)^J}  \\
&+ \#\CA(X) \exp\bigg( - \frac{c_g^{1/2}}{100} (\log X)^{1/2}  \bigg) . 
\end{align*}
Now, we put $m\mu$ into arithmetic progressions. To do this, we note that given $f\mod{g^J}$ with $f=\sum_{i=0}^{J-1}a_i g^i$ such that $a_i\not=b$ for all $i$, we are looking at $m\mu\equiv f \mod{g^J}$. Hence,
\begin{align*}
\sum_{m\leq \sqrt{X}} \frac{r_0(m)}{m} &\sum_{\substack{\mu \mod{g^J} \\ m\mu \equiv \sum_{i=0}^{J-1}a_ig^i \mod{g^J} \text{ with } a_i\not=b, \forall i}} \frac{r(\mu ; g^J)}{g^J(g-1)^J}  \\
&= \sum_{\substack{f \mod{g^J} \\ f = \sum_{i=0}^{J-1}a_ig^i \text{ with } a_i\not=b, \forall i}} \sum_{m\leq \sqrt{X}} \frac{r_0(m)}{m} \sum_{\substack{\mu\mod{g^J} \\ m\mu \equiv f \mod{g^J}}} \frac{r(\mu ; g^J)}{g^J(g-1)^J} \\
&=  \sum_{\substack{f \mod{g^J} \\ f = \sum_{i=0}^{J-1}a_ig^i \text{ with } a_i\not=b, \forall i}} \sum_{v\mod{g^J}} \sum_{\substack{\mu\mod{g^J} \\ v\mu\equiv f \mod{g^J}}} \frac{r(\mu ; g^J)}{g^J(g-1)^J} \sum_{\substack{m\leq \sqrt{X} \\ m\equiv v \mod{g^J}}} \frac{r_0(m)}{m}.
\end{align*}
The inner sum is $\ll \frac{r(v;g^J)}{g^{2J}} \log X$ by partial summation and by suming Lemma \ref{sum_of_squares_in_short_int_and_aps} over the representations of $v$ as sum of two squares $\mod{g^J}$, and since $g^J \ll_{\varepsilon} X^{\varepsilon}$. Hence, we get that this is in total
$$ \ll (\log X) \bigg( \frac{g^J}{(g-1)^J} \sum_{\substack{f \mod{g^J} \\ f = \sum_{i=0}^{J-1}a_ig^i \text{ with } a_i\not=b, \forall i}} \frac{1}{g^{4J}} \sum_{\substack{v , \mu\mod{g^J} \\ v\mu\equiv f \mod{g^J}}} r(\mu ; g^J) r(v ; g^J) \bigg) . $$
Finally, we want to understand the part with $J$. We expect it to converge to a constant as $J\to\infty$. To see this, we will do some local analysis. Firstly, we factorize $g=p_1^{e_1}\cdots p_{\ell}^{e_{\ell}}$. Then, by the Chinese Remainder Theorem, we have
\begin{align*}
&\frac{g^J}{(g-1)^J} \sum_{\substack{f \mod{g^J} \\ f = \sum_{i=0}^{J-1}a_ig^i \text{ with } a_i\not=b, \forall i}} \frac{1}{g^{4J}} \sum_{\substack{v , \mu\mod{g^J} \\ v\mu\equiv f \mod{g^J}}} r(\mu ; g^J) r(v ; g^J) \\
&=  \frac{g^J}{(g-1)^J} \sum_{\substack{f \mod{g^J} \\ f = \sum_{i=0}^{J-1}a_ig^i \text{ with } a_i\not=b, \forall i}} \prod_{i=1}^{\ell} \bigg( \frac{1}{p_i^{4Je_i}} \sum_{\substack{v_i,\mu_i \mod{p_i^{Je_i}} \\ v_i \mu_i \equiv f \mod{p_i^{Je_i}}}} r(\mu_i ; p_i^{Je_i}) r(v_i ; p_i^{Je_i})  \bigg) .
\end{align*}
Now, we use the trivial bound $r(\nu;p^k) \leq 2 p^k$ for $p$ prime and $k\geq 1$. This gives us
$$ \ll_g  \frac{g^J}{(g-1)^J} \sum_{\substack{f \mod{g^J} \\ f = \sum_{i=0}^{J-1}a_ig^i \text{ with } a_i\not=b, \forall i}} \prod_{i=1}^{\ell} \bigg( \frac{1}{p_i^{2Je_i}} \sum_{\substack{v_i,\mu_i \mod{p_i^{Je_i}} \\ v_i \mu_i \equiv f \mod{p_i^{Je_i}}}} 1 \bigg), $$
and we can bound the inner sum by $\leq p^{Je_i}$ since it is either $0$ or $p^{Je_i}$. Therefore, taking the product, we get
$$ \ll_g  \frac{1}{(g-1)^J} \sum_{\substack{f \mod{g^J} \\ f = \sum_{i=0}^{J-1}a_ig^i \text{ with } a_i\not=b, \forall i}} 1 = 1,  $$
since the sum over $f$ is $(g-1)^J$.
Therefore, we get
$$ \sum_{n\in \CA(X)} (r_*^2(n) - 2 r_*(n)) \ll_g \frac{\#\CA(X)(\log\log X)^4}{\log X} , $$
which proves Theorem \ref{off-diagonal_result}.

\nocite{*}
\bibliographystyle{plain}
\bibliography{Missing_digits_bib}

\end{document}